\documentclass[11pt,reqno]{amsart}
\usepackage[usenames]{color}
\usepackage{amsmath,pdfsync,verbatim,graphicx,epstopdf,enumerate}
\usepackage{ulem}
\newtheorem{theorem}{Theorem}[section]

\newtheorem{lemma}[theorem]{Lemma}
\newtheorem{proposition}[theorem]{Proposition}

\theoremstyle{definition}

\usepackage{amsmath,amssymb,verbatim,amsthm}
\usepackage{moreverb}
\usepackage{amsmath,amstext,amssymb,latexsym,amscd,epsfig,amsthm}
\usepackage{epsfig,graphicx,graphics,subfig,pdfpages,pifont}
\graphicspath{{./figures/}}
\usepackage{epstopdf,bm}
\usepackage{fontenc}

\usepackage{setspace}
\doublespace

\newcommand{\D}{\mathrm{d}}

\newcommand{\lb}{\left(}

\newcommand{\vp}{\varphi}
\newcommand{\ve}{\varepsilon}

\newcommand{\rb}{\right)}
\newcommand{\PD}{\partial}

\newcommand{\wt}{\widetilde}

\newcommand{\Oc}{\mathcal{O}}

\newcommand{\Rb}{\mathbb{R}}
\newcommand{\Sb}{\mathbb{S}}

\newcommand{\Beq}{\begin{equation}}
\newcommand{\Eeq}{\end{equation}}
\newcommand{\beq}{\begin{equation*}}
\newcommand{\eeq}{\end{equation*}}
\newcommand{\bal}{\begin{align}}
\newcommand{\eal}{\end{align}}
\renewcommand{\O}{\Omega}

\newcommand{\n}{\nabla}

\newcommand{\A}{\alpha}
\newcommand{\B}{\beta}
\newcommand{\bp}{\begin{prob}}
\newcommand{\ep}{\end{prob}}
\newcommand{\bpr}{\begin{proof}}
\newcommand{\epr}{\end{proof}}

\definecolor{wine}{RGB}{165,0,33}
\definecolor{greelue}{RGB}{0,180,180}

\newcommand{\bel}[1]{\begin{equation}\label{#1}}
\newcommand{\ee}{\end{equation}}

\title{Image reconstruction from radially incomplete spherical Radon data}
\author[Ambartsoumian, Gouia-Zarrad, Krishnan and Roy]{Gaik Ambartsoumian$^\ast$, Rim Gouia-Zarrad$^\dagger$, Venkateswaran P. Krishnan$^\diamond$ and Souvik Roy$^\sharp$}
\address{$^{\ast}$Corresponding author, Department of Mathematics, University of Texas at Arlington, USA \newline\indent\:  E-mail:{\tt gambarts@uta.edu}
\newline
\indent $^{\dagger}$ Department of Mathematics and Statistics, American University of Sharjah, UAE, E-mail:{\tt rgouia@aus.edu}
\newline
\indent $^{\diamond}$Tata Institute of Fundamental Research - Centre for Applicable Mathematics, Bangalore, India
\newline
\indent\: E-mail:{\tt vkrishnan@math.tifrbng.res.in}
\newline
\indent $^{\sharp}$Department of Mathematics, University of W\"urzburg, Germany
 E-mail:{\tt souvik.roy@mathematik.uni-wuerzburg.de}}
\begin{document}


\begin{abstract}
We study inversion of the spherical Radon transform with centers on a sphere (the data acquisition set). Such inversions are essential in various image reconstruction problems arising in medical, radar and sonar imaging. In the case of radially incomplete data, we show that the spherical Radon transform can be uniquely inverted recovering the image function in spherical shells. Our result is valid when the support of the image function is inside the data acquisition sphere, outside that sphere, as well as on both sides of the sphere. Furthermore, in addition to the uniqueness result our method of proof provides reconstruction formulas for all those cases. We present a robust computational algorithm and demonstrate its accuracy and efficiency on several numerical examples.
\end{abstract}
\maketitle


\section{Introduction}
The spherical Radon transform (SRT) maps a function of $n$ variables to its integrals over a family of spheres in $\mathbb{R}^n$. Such transforms naturally appear in mathematical models of various imaging modalities in medicine \cite{AZSZP, HKN, Kuchment-Kunyansky_2008, Mensah-Franceschini, Norton-Circular, Norton-Linzer, RKCV,Stefanov-Uhlmann-TAT,Stefanov-Uhlmann-brain-imaging, Xu-Wang}, geophysical applications \cite{deHoop,Louis-Quinto}, radar \cite{Cheney-Borden-Book}, as well as in some purely mathematical problems of approximation theory \cite{ABK-Spherical-Lp, Agranovsky-Quinto, Lin-Pinkus}, PDEs \cite{ABK-Spherical-Lp, AKQ-Spherical,FHR-Spherical-Even,FPR-Spherical-Odd,Finch-Rakesh,John-Book,Kunyansky-Spherical-1,Kunyansky-Spherical-2} and integral geometry \cite{AGL-Circular,Ambartsoumian-Krishnan, AK1,AK2,Andersson_1988,AER,GGG-Book,Haltmeier,Nguyen, Rubin,Salman}.

One of the most important questions related to SRT is the possibility of its stable inversion. Since the family of all spheres in $\mathbb{R}^n$ has $n+1$ dimensions, the problem of inversion from the set of integrals along all spheres is overdetermined. Hence it is customary to consider the problem of inverting the SRT from the restriction of the full set of integrals to an $n$-dimensional subset. While one can come up with several different choices of such subsets, a common approach (especially in imaging applications) is to restrict the centers of integration spheres to a hypersurface in $\mathbb{R}^n$.

For example, a simple model of thermoacoustic tomography (TAT) can be described as follows. A biological object under investigation is irradiated with a short pulse of electromagnetic waves. Certain part of that radiation gets absorbed in the body heating up the tissue leading to its thermoelastic expansion. The latter generates ultrasound waves, which propagate through the body and are registered by transducers placed on its surface. Under a simplifying assumption of constant speed $c$ of ultrasound waves in the tissue, at any moment of time $t$, a single transducer records a superposition of signals generated at locations that are at the fixed distance $ct$ from the transducer. In other words, the transducer measurements can be modeled as  integrals of a function along spheres centered at the transducer location and of different radii (depending on time). By moving the transducer around the surface of the object (or equivalently using an array of such transducers) one can essentially measure a 3-dimensional family of spherical integrals of the unknown image function. Hence to recover the image in this simple TAT model, one would need to invert the SRT in the setup described above. Similar mathematical problems arise also in various models of ultrasound reflection tomography, as well as in sonar and radar imaging.

While our work is motivated by its potential applications in imaging problems,    we study the spherical Radon transform in $\Rb^{n}$ for any $n\geq 3$.  We discuss the inversion of SRT from integrals of a function $f$ along spheres whose centers lie on the surface of the unit (data acquisition) sphere\footnote{Our results carry over with little difficulty when the centers of the SRT data lie on a sphere of radius $R$.}.  With the additional restriction on the set of radii of integration spheres, we prove the uniqueness as well as derive reconstruction formulas for $f$ from such data.  We provide several results that hold for the cases when the support of $f$ is inside, outside, or on both sides of the unit sphere. More precisely, for the case when the support of a function $f$ is inside the unit sphere, our result shows that in order to reconstruct $f$ in the spherical shell
$\{r<|x|<1\}$ for any $r<1$, we only need SRT data with centers on the unit sphere and for all radii $\rho$ such that $0<\rho<1-r$. Analogous statements can be made for the case when the support of $f$ is outside or on both sides of the unit sphere. In connection with this, we mention the result \cite[Theorem 5]{FPR-Spherical-Odd}, where it was shown that for a bounded open connected set $D$ in $\Rb^{n}$ for $n$ odd, a function $f$ supported in $\overline{D}$ can be reconstructed from SRT data with centers on $\PD D$ and all radii $\rho$ such that $\rho\in[0,\mbox{diam}(D)/2]$. One of the consequences of our work is a generalization of this result for the case of even dimensions, as well as when the support of the function lies inside, outside or on both sides for the case when $D$ is a sphere. We emphasize here that the uniqueness result [17, Theorem 5] was already generalized for variable sound speeds in \cite{Stefanov-Uhlmann-TAT} in all space dimensions; see Prop. 2 in that paper. If one is interested in uniqueness results alone, unique continuation arguments as in \cite{Stefanov-Uhlmann-TAT}  or analytic microlocal analysis methods  as in \cite{Agranovsky-Quinto-Duke, Quinto-Spherical-2} could be used\footnote{We thank Plamen Stefanov for bringing this as well as the result stated in the previous sentence to our attention.}, although to the best of our knowledge, even for the case of spherical acquisition surface and for functions supported outside or on both sides of the sphere, such results have not been published.  The advantage of our work, in the specific setting where the acquisition geometry is the unit sphere, is that it provides in addition to uniqueness results, inversion formulas using radially partial data.

The paper is organized as follows. The main results are stated in Section \ref{main_results} and the proofs are presented in Section \ref{proofs}. In Section \ref{3d} we write down the inversion formulas for the special case of $n=3$. In Section \ref{numerical_algorithms} we discuss the numerical algorithm based on the product integration method. In Section \ref{num_res}, we provide numerical examples illustrating the accuracy and efficiency of the proposed inversion algorithms.

\section{Main results}\label{main_results}
We consider the usual spherical coordinate system:
\begin{align*}
& x_{1}=r\cos\vp_{1}\\
&x_{2}=r\sin\vp_{1}\cos\vp_{2}\\
&x_{3}=r\sin\vp_{1}\sin\vp_{2}\cos\vp_{3}\\
& \vdots\\
&x_{n-1}=r\sin\vp_{1}\sin\vp_{2}\cdots\sin\vp_{n-2}\cos\vp_{n-1}\\
&x_{n}=r\sin\vp_{1}\sin\vp_{2}\cdots\sin\vp_{n-2}\sin\vp_{n-1},
\end{align*}
where $0\leq \vp_{i}\leq \pi$ for $1\leq i\leq n-2$ and $0\leq \vp_{n-1}\leq 2\pi$. For simplicity, from now on, we will denote $\vp=(\vp_{1},\cdots,\vp_{n-1})$.
Let us consider the unit sphere centered at the origin in $\Rb^{n}$ and fix an arbitrary point $C$ on this sphere. We will denote $C$ in the above spherical coordinates by $\A$, where $\A=(\A_{1},\cdots,\A_{n-1})$.  The Cartesian coordinates of the point $C$ will then be
\[
(\cos \A_{1},\sin \A_{1}\cos \A_{2},\cdots, \sin\A_{1}\sin\A_{2}\cdots\sin\A_{n-2}\cos \A_{n-1},\sin\A_{1}\sin\A_{2}\cdots\sin\A_{n-2}\cos \A_{n-1}).
\]

Let $f:\mathbb{R}^{n}\to \mathbb{R}$ be a continuous function of compact support. Consider a sphere $S(\rho,\A)$ of radius $\rho$ centered at $C$. The spherical Radon transform of $f$ along the sphere $S(\rho,\A)$  for $\rho>0$ and $\A=(\A_{1},\cdots,\A_{n-2},\A_{n-1})\in [0,\pi]\times \cdots\times[0,\pi]\times[0,2\pi]$ is defined as

\begin{equation}\label{spher-transf-def}
Rf(\rho,\A)=g(\rho,\A)=\int_{S(\rho,\A)} f\, \D \Omega,
\end{equation}
where $\D \O$ is the usual surface measure on the sphere $S(\rho,\A)$.

 Finally, let us denote by $A(R_{1},R_{2})$ the spherical shell lying between the spheres of radii $R_1$ and $R_2$ centered at the origin. Expressing in Cartesian coordinates:
\[
A(R_{1},R_{2})=\{x\in \Rb^{n}: R_{1}<|x|<R_{2}\}.
\]

Note that this spherical shell can also be expressed in spherical coordinates by
\[
A(R_{1},R_{2})=\{(r,\vp): R_{1}<r<R_{2}, 0\leq \vp_{i}\leq \pi \mbox{ for } 1\leq i\leq n-2 \mbox{ and } 0\leq \vp_{n-1}\leq 2\pi\}.
\]

We now state the main results.

\begin{theorem}[Exterior support]\label{thm:ext}
Let $f(r,\vp)$ be a $C^{\infty}$ function supported inside $A(1,\,3)$. If $Rf(\rho,\A)$ is
known for all $(\rho,\A)$ with $0<\rho<R_{1}$ where $0<R_1<2$ and $\A\in[0,\pi]\times \cdots\times [0,\pi]\times [0,2\pi]$, then $f(r,\vp)$ can be uniquely recovered in the spherical shell $A(1,\,1+R_1)$ with an iterative reconstruction procedure.
\end{theorem}

\begin{theorem}[Interior support]\label{thm:int}
Let $f(r,\vp)$ be a $C^{\infty}$ function supported inside $A(\ve,1)$. If $Rf(\rho,\A)$ is
known for all $(\rho,\A)$ with $0<\rho<1-\ve$, where $0<\ve<1$ and $\A\in[0,\pi]\times \cdots\times [0,\pi]\times [0,2\pi]$, then $f(r,\vp)$ can be uniquely recovered in the spherical shell $A(\ve, 1)$ with an iterative reconstruction procedure.
\end{theorem}
\begin{theorem}[Interior and exterior support]\label{thm:int-ext}
Let $f(r,\varphi)$ be a $C^{\infty}$ function supported inside the ball  $B(0,R_{2})$ centered at the origin and of radius $R_{2}>2$. Define $R_{1}=R_{2}-2$. If $Rf(\rho,\alpha)$ is
known for all $\rho$ with $R_{2}-1<\rho<R_{2}+1$ and $\A\in[0,\pi]\times \cdots\times [0,\pi]\times [0,2\pi]$, then $f(r,\varphi)$ can be uniquely recovered in the spherical shell $A(R_{1},R_{2})$ with an iterative reconstruction procedure.
\end{theorem}

\begin{figure}[h]
  \centering
    \subfloat[]{%
    \label{sketch_ext}\includegraphics[height=50mm,keepaspectratio]{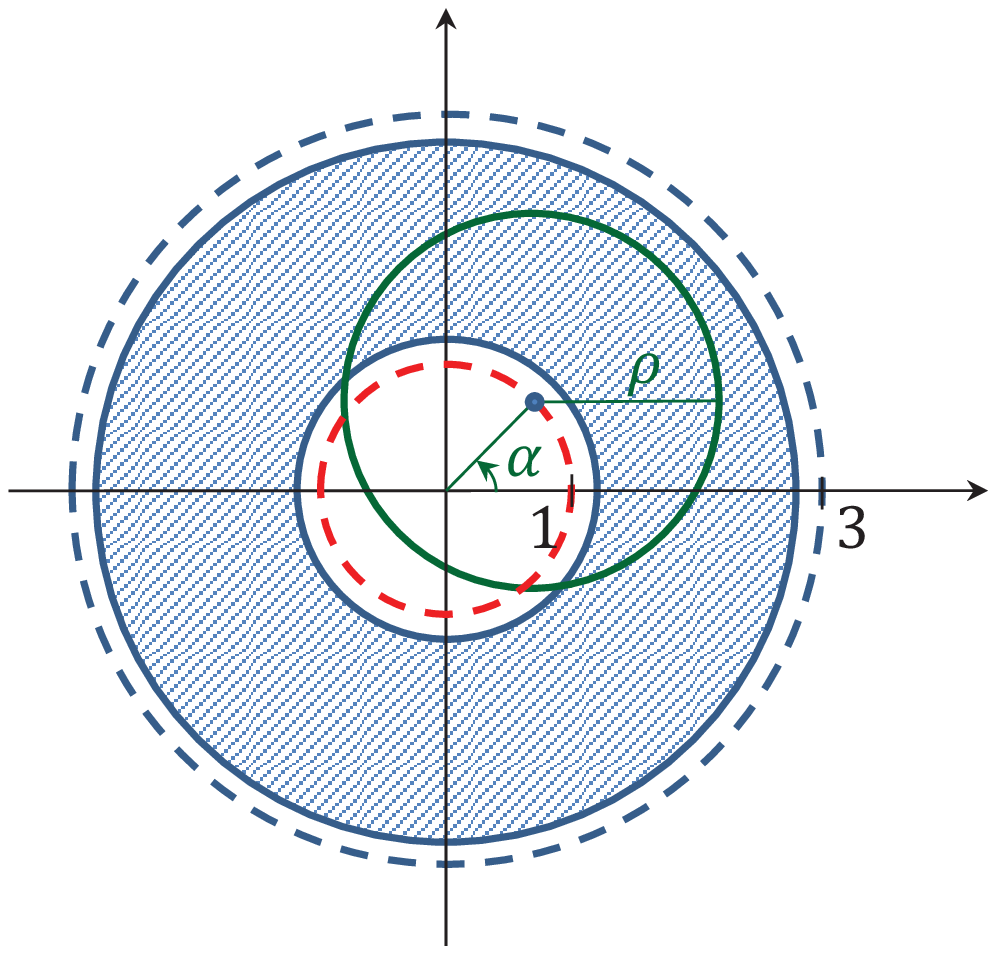}
               }
    \subfloat[]{%
     \label{sketch_int} \includegraphics[height=50mm,keepaspectratio]{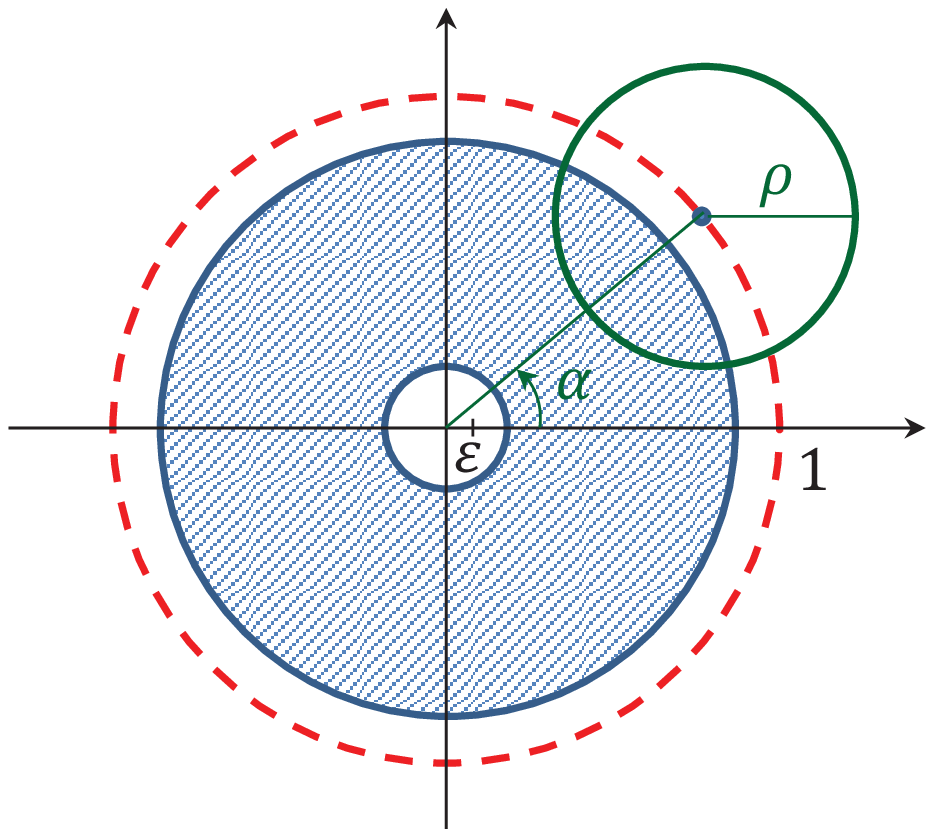}
               }
    \subfloat[]{%
    \label{sketch_int_ext}\includegraphics[height=50mm,keepaspectratio]{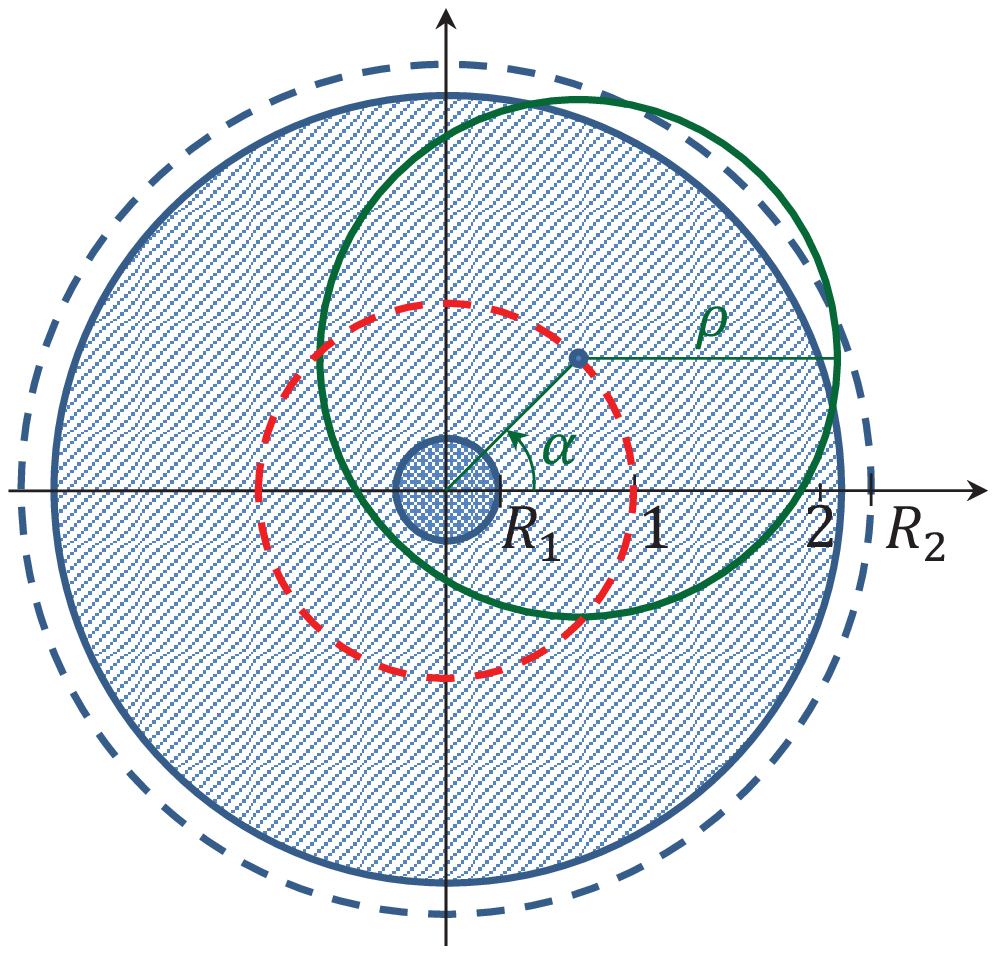}
               }
    \caption{Sketches illustrating the setups of Theorems \ref{thm:ext} in \eqref{sketch_ext}, \ref{thm:int}  in \eqref{sketch_int}, \ref{thm:int-ext}  in \eqref{sketch_int_ext}. The shaded area contains the support of $f(r,\vp)$ and the smaller dashed circle represents the data acquisition surface.}
 \label{sketches}
\end{figure}

\section{Proofs}\label{proofs}
Let $\{Y_{l}\}$ be the full set of spherical harmonics forming an orthonormal basis for $L^2$ functions on $\Sb^{n-1}$. We expand $f$ and $g$ into a series involving $\{Y_{l}\}$. We have
\begin{align}
& \label{f-expansion}
f(r,\vp)=\sum_{l={0}}^{\infty }f_{l}(r)\,Y_{l}(\vp)\\
&\label{g-expansion} g(\rho,\alpha)=\sum_{l={0}}^{\infty }g_{l}(\rho)\,Y_{l}(\alpha).
\end{align}

Due to rotational invariance of the spherical Radon transform, the spherical harmonics expansion of $f$ and $g$ leads to diagonalization of the transform, that is, for each fixed $l\ge0$ the coefficient $g_l(\rho)$ depends only on $f_l(r)$. Our main goal in the following calculations is to find that relationship, and express $f_l(r)$ through $g_l(\rho)$.

Using \eqref{f-expansion} in \eqref{spher-transf-def}, the spherical Radon transform is expressed as
\[g(\rho,\alpha)=\int\limits_{S(\rho,\alpha)} f\, \D \Omega=\int\limits_{S(\rho,\alpha)}\sum_{l={0}}^{\infty }f_{l}(r)\,Y_{l}(\vp)\, \D \Omega.
\]

Since $f$ is a $C^{\infty}$ function of compact support, by straightforward modifications of the arguments in \cite{Kalf_Paper}\footnote{The result in \cite{Kalf_Paper} shows that the spherical harmonics series of a sufficiently smooth function $h$ on the unit sphere converges uniformly to $h$. One can adapt the same arguments to show that the spherical harmonics series of a compactly supported smooth function $f$ on $(0,\infty)\times \Sb^{n-1}$ converges uniformly to $f$ in the radial and angular variables.}, we have that the spherical harmonics series of $f$ converges uniformly to $f$. Hence we can interchange the sum and the integral, and we have
\begin{equation}\label{Spherical Radon transform}
g(\rho,\alpha)=\sum_{l={0}}^{\infty }\int\limits_{S(\rho,\A)}f_{l}(r)\,Y_{l}(\vp)\, \D \Omega.
\end{equation}

We denote by $\vec{C}_{1}$ the vector pointing from the origin to the fixed point $C$ on the unit sphere in $\Rb^{n}$. Let us fix an orthonormal coordinate system for the plane $\vec{C}_{1}^{\perp}$, which we denote by
$\vec{C}_{2},\cdots,\vec{C}_{n}$. Reordering the vectors if necessary, we assume that $(\vec{C}_{1},\cdots,\vec{C}_{n-1},\vec{C}_{n})$ is an oriented orthonormal coordinate system for $\Rb^{n}$. We can consider spherical coordinates with respect to this coordinate system, and denote them by $(\wt{r},\wt{\vp})$ where $\wt{\vp}=(\wt{\vp}_{1},\cdots,\wt{\vp}_{n-1})$.


The surface measure $\D \O$ on the sphere $S(\rho,\A)$ in this  coordinate system is
\Beq\label{Surface measure new coordinates}
d\O=\rho^{n-1}\sin^{n-2}\wt{\vp}_{1}\sin^{n-3}\wt{\vp}_{2}\cdots\sin^{2}\wt{\vp}_{n-3}\sin\wt{\vp}_{n-2}\D\wt{\vp}_{1}\cdots \D\wt{\vp}_{n-1}.
\Eeq

Our next goal, which we state as Proposition \ref{prop:Transformation} below, is to find the above surface measure $\D \O$ in the spherical coordinate system $(r,\vp)$. To this end, we first prove a lemma.

Let us consider an arbitrary point $P$ in the coordinate system $(r,\vp)$ and denote it as  $r\vec{P}$ with $\vec{P}\in \Sb^{n-1}$. We recall the orthonormal coordinate system $(\vec{C}_{1},\cdots,\vec{C}_{n-1},\vec{C}_{n})$ with respect to the arbitrary fixed point $C\in \Sb^{n-1}$ introduced above and define
\Beq\label{eq:Ai}
A_{i}=\vec{P}\cdot \vec{C}_{i} \mbox{ for } 1\leq i\leq n.
\Eeq

\begin{lemma}\label{Jacobian lemma}
We have the following formula:
\Beq\label{Matrix determinant}
\det
\begin{pmatrix}
\n_{\vp}A_{1}\\
\vdots\\
\n_{\vp}A_{n-1}
\end{pmatrix}=\lb \vec{P}\cdot \vec{C}_{n}\rb \sin^{n-2}\vp_{1}\sin^{n-3}\vp_{2}\cdots\sin^{2}\vp_{n-3}\sin\vp_{n-2}\D{\vp}_{1}\cdots  \D{\vp}_{n-1}.
\Eeq
\end{lemma}
The proof of the above lemma relies on the following result due to Cauchy and Binet.
\begin{theorem}
[Cauchy-Binet]
Let $A$ be an $m\times n$ and $B$ be an $n\times m$ matrix. Then
\[
\det(AB)=\sum\limits_{J}\det(A(J))\det(B(J))
\]
with \[
J=j_{1},j_{2},\cdots,j_{m},\quad 1\leq j_{1}<j_{2}\cdots<j_{k}\leq m
\]
and $A(J)$ denotes the matrix formed from $A$ with the columns $J$ with the order preserved and $B(J)$ denotes the matrix formed from $B$ with the rows $J$ with the order preserved.
\end{theorem}

\begin{proof}[Proof of Lemma \ref{Jacobian lemma}]
Since $A_{i}=\vec{P}\cdot \vec{C}_{i}$, we have that
\[
\begin{pmatrix}
\n_{\vp}A_{1}\\
\vdots\\
\n_{\vp}A_{n-1}
\end{pmatrix}=\begin{pmatrix}
\vec{C}_{1}\\
\vdots\\
\vec{C}_{n-1}
\end{pmatrix}\begin{pmatrix}
\frac{\PD\vec{P}^{t}}{\PD \vp_{1}}\cdots \frac{\PD\vec{P}^{t}}{\PD \vp_{n-1}}
\end{pmatrix}.
\]
We are interested in calculating the determinant of $(n-1)\times (n-1)$ matrix that is written as a product of $(n-1)\times n$ matrix with an $n\times(n-1)$ matrix (the first matrix comprising of $\vec{C}_{i}$ and the second one involving the derivatives with respect to $\vp$ of $\vec{P}$).


We have
\[
\begin{pmatrix}
\frac{\PD\vec{P}^{t}}{\PD \vp_{1}}\cdots \frac{\PD\vec{P}^{t}}{\PD \vp_{n-1}}
\end{pmatrix}=
\]
{\footnotesize
\[
\begin{pmatrix}
-\sin\vp_{1}&0&\cdots&\cdots &0\\
\cos\vp_{1}\cos\vp_{2}&-\sin\vp_{1}\sin\vp_{2}&\cdots&\cdots &0\\
\cos\vp_{1}\sin\vp_{2}\cos\vp_{3}&\sin\vp_{1}\cos\vp_{2}\cos\vp_{3}&-\sin\vp_{1}\sin\vp_{2}\sin\vp_{3}&\cdots&0\\
\vdots&\vdots&\vdots&\ddots&0\\
\cos\vp_{1}\sin\vp_{2}\cdots\sin\vp_{n-2}\cos\vp_{n-1}&\cdots&\cdots&\cdots&-\sin\vp_{1}\cdots\sin\vp_{n-1}\\
\cos\vp_{1}\sin\vp_{2}\cdots\sin\vp_{n-1}&\cdots&\cdots&\cdots&\sin\vp_{1}\cdots\sin\vp_{n-2}\cos\vp_{n-1}
\end{pmatrix}.
\]
}
The determinant of this matrix is
\[
=\sin^{n-2}\vp_{1}\sin^{n-3}\vp_{2}\cdots\sin\vp_{n-2}\det
\begin{pmatrix}
v_{1}&v_{2}&\cdots&v_{n-1}
\end{pmatrix}
\]
with the vectors $v_{i}$ for $1\leq i\leq n-1$ being an orthonormal collection of $n-1$ vectors perpendicular to the vector $\vec{P}$. Note that each of these vectors is perpendicular to $\vec{P}$ because each $v_{i}$ is obtained by differentiating $\vec{P}$ with respect to $\vp_{i}$.

Now we have
\[
\det \begin{pmatrix}v_{1}&\cdots& v_{n-1}&\vec{P}\end{pmatrix}=\pm 1,
\]
since the matrix belongs to $O(n)$. We can write the above determinant as
\[
\sum\limits_{i=1}^{n}(-1)^{n+i}\vec{P}_{i}\cdot M_{in}=\pm 1,
\]
where $\vec{P}_{i}$ denotes the $i^{\mathrm{th}}$ component of $i$ and $M_{in}$ denotes the corresponding minor.

Since $\begin{pmatrix}v_{1}&\cdots& v_{n-1}&\vec{P}\end{pmatrix}\in O(n)$, this implies that
\[
\lb (-1)^{1+n}M_{1n},\cdots,(-1)^{2n}M_{nn}\rb =\pm \vec{P}.
\]

Since $\vec{C}_{i}$ for $1\leq i\leq n-1$ are orthonormal and oriented, we have that
\[
\det\begin{pmatrix}
\vec{C}_{1}\\
\vdots\\
\vec{C}_{n-1}\\
\vec{C}_{n}
\end{pmatrix}= 1.
\]
The same argument as above shows that the vector with the minors $\wt{M}_{ni}$ coming from this matrix satisfies
\[
\lb (-1)^{1+n}\wt{M}_{n1},\cdots,(-1)^{2n}\wt{M}_{nn}\rb =\pm \vec{C}_{n}.
\]
Now using Cauchy-Binet theorem,  \eqref{Matrix determinant} is proved.
\end{proof}

Now we find the surface measure \eqref{Surface measure new coordinates} with respect to the coordinate system $(r,\vp)$.
\begin{proposition}\label{prop:Transformation} The surface measure $\D \O$ on the sphere $S(\rho,\A)$ with respect to the spherical coordinate system $(r,\vp)$ is given by
\[
\D \O = \frac{\rho^{n-2}r^{2}}{\left|r-A_{1}\right|} \sin^{n-2}\vp_{1}\cdots\sin\vp_{n-2} \D\vp_{1}\cdots \D\vp_{n-1},
\]
where $A_{1}$ is defined in \eqref{eq:Ai}.
\end{proposition}

\begin{proof}
We have
\[
d\O=\rho^{n-1}\sin^{n-2}\wt{\vp}_{1}\sin^{n-3}\wt{\vp}_{2}\cdots\sin^{2}\wt{\vp}_{n-3}\sin\wt{\vp}_{n-2}\D\wt{\vp}_{1}\cdots \D\wt{\vp}_{n-1}.
\]
We express $\cos\wt{\vp}_{i}$ for $1\leq i\leq n-1$ in terms of the coordinates $(\vp_{1},\cdots,\vp_{n-1})$.

We have:
\[
\cos \wt{\vp}_{1}=\frac{\lb r\vec{P}_{1}-\vec{C}_{1}\rb\cdot \vec{C}_{1}}{|r\vec{P}-\vec{C}_{1}|}=\frac{r\vec{P}_{1}\cdot \vec{C}_{1}-1}{\rho}=
\frac{rA_1-1}{\rho}.
\]

Furthermore it is easy to see that
\[
\begin{pmatrix}
\cos\wt{\vp}_{1}\\
\vdots\\
\cos \wt{\vp}_{n-1}
\end{pmatrix}=
\begin{pmatrix}
\frac{rA_{1}-1}{\rho}\\
\frac{A_{2}}{\sqrt{1-A_{1}^{2}}}\\
\vdots\\
\frac{A_{n-1}}{\sqrt{1-\lb A_{1}^{2}+\cdots +A_{n-2}^{2}\rb}}
\end{pmatrix}.
\]
Let us compute the determinant of the Jacobian of the transformation
\Beq\label{Transformation}
(\vp_{1},\cdots,\vp_{n-1})\to (\cos\wt{\vp}_{1},\cdots,\cos\wt{\vp}_{n-1}).
\Eeq
Since
\begin{equation}\label{cos-th}
\rho^{2}= r^{2}+1-2rA_{1},
\end{equation}
differentiating this  equation, we get
\[
\frac{\PD r}{\PD \vp_{i}}=\frac{r}{r-A_{1}}\frac{\PD A_{1}}{\PD \vp_{i}}.
\]

The Jacobian matrix of \eqref{Transformation} is
\[
\begin{pmatrix}
\frac{r^{2}}{\rho(r-A_{1})}\nabla_{\vp} A_{1}\\
\frac{1}{\sqrt{1-A_{1}^{2}}}\n_{\vp}A_{2}+\frac{A_{1}A_{2}}{(1-A_{1}^{2})^{3/2}}\n_{\vp}A_{1}\\
\vdots\\
\frac{1}{\sqrt{1-(A_{1}^{2}+\cdots+A_{n-2}^{2})}}\n_{\vp}A_{n-1}+\frac{A_{1}A_{n-1}\n_{\vp}A_{1}+\cdots+A_{n-2}A_{n-1}\n_{\vp}A_{n-1}}{(1-(A_{1}^{2}+\cdots+A_{n-2}^{2})^{3/2}}\end{pmatrix}
\]
Here $\n_{\vp}$ denotes the $(n-1)$-vector $(\frac{\PD}{\PD \vp_{1}},\cdots,\frac{\PD}{\PD \vp_{n-1}})$.

The determinant of the matrix above is the same as the determinant of the matrix
\[\frac{r^{2}}{\rho(r-A_{1})}\frac{1}{\sqrt{1-A_{1}^{2}}}\cdots\frac{1}{\sqrt{1-(A_{1}^{2}+\cdots+A_{n-2}^{2}}}
\begin{pmatrix}
\n_{\vp}A_{1}\\
\vdots\\
\n_{\vp}A_{n-1}
\end{pmatrix}.
\]
Recall that  we are interested in expressing
\[
\sin^{n-2}\wt{\vp}_{1}\cdots \sin\wt{\vp}_{n-2}\D\wt{\vp}_{1}\cdots \D\wt{\vp}_{n-1}
\]
in terms of $\D{\vp}_{1}\cdots \D{\vp}_{n-1}$.
Using Lemma \ref{Jacobian lemma}, we have,
\begin{align}\label{Measure transformation}
\notag \lb \sin\wt{\vp}_{1}\cdots\sin\wt{\vp}_{n-1}\rb \D\wt{\vp}_{1}\cdots \D\wt{\vp}_{n-1}=&\frac{r^{2}}{\rho(r-A_{1})}\frac{1}{\sqrt{1-A_{1}^{2}}}\cdots\frac{1}{\sqrt{1-(A_{1}^{2}+\cdots+A_{n-2}^{2})}}\\
 &\times \lb \vec{P}\cdot \vec{C}_{n}\rb \sin^{n-2}\vp_{1}\cdots\sin\vp_{n-2} \D\vp_{1}\cdots \D\vp_{n-1}
\end{align}
Note that
\[
\sin\wt{\vp}_{n-1}=\frac{\vec{P}\cdot \vec{C}_{n}}{\sqrt{1-(A_{1}^{2}+\cdots+A_{n-2}^{2})}}.
\]
Therefore we have
\begin{align*}
\lb \sin\wt{\vp}_{1}\cdots\sin\wt{\vp}_{n-2}\rb \D\wt{\vp}_{1}\cdots \D\wt{\vp}_{n-1}&=\frac{r^{2}}{\rho(r-A_{1})}\frac{1}{\sqrt{1-A_{1}^{2}}}\cdots\frac{1}{\sqrt{1-(A_{1}^{2}+\cdots+A_{n-3}^{2})}}\\
&\quad \times \sin^{n-2}\vp_{1}\cdots\sin\vp_{n-2} \D\vp_{1}\cdots \D\vp_{n-1}.
\end{align*}
Hence
\begin{align}\label{measure change of variable formula}
\notag &{\sqrt{1-A_{1}^{2}}}\cdots{\sqrt{1-(A_{1}^{2}+\cdots+A_{n-3}^{2})}}\lb \sin\wt{\vp}_{1}\cdots\sin\wt{\vp}_{n-2}\rb d\wt{\vp}_{1}\cdots \wt{\vp}_{n-1}\\
&\quad =\frac{r^{2}}{\rho(r-A_{1})} \sin^{n-2}\vp_{1}\cdots\sin\vp_{n-2} \D\vp_{1}\cdots \D\vp_{n-1}.
\end{align}

Now we have
\[
\left|\frac{\sqrt{1-(A_{1}^{2}+\cdots+A_{n-3}^{2})}}{\sqrt{1-(A_{1}^{2}+\cdots+A_{n-4}^{2})}}\right|=\left|\sin\wt{\vp}_{n-3}\right|.
\]
Multiplying and dividing the left hand side of \eqref{measure change of variable formula}, by $\sqrt{1-(A_{1}^{2}+\cdots+A_{n-4}^{2})}$ and then by $(1-(A_{1}^{2}+\cdots+A_{n-5}^{2}))$ and continuing this way, we get

\begin{align}\notag&\sin^{n-2}\wt{\vp}_{1}\sin^{n-3}\wt{\vp}_{2}\cdots\sin\wt{\vp}_{n-2} \D\wt{\vp}_{1}\cdots \D\wt{\vp}_{n-1}\\
&\quad=\frac{r^{2}}{\rho(r-A_{1})}\sin^{n-2}\vp_{1}\cdots\sin\vp_{n-2} \D\vp_{1}\cdots \D\vp_{n-1}.
\end{align}

Since we are interested in the absolute value of the determinant of the Jacobian of the transformation in \eqref{Transformation}, we finally have
\begin{align*}
&\rho^{n-1}\sin^{n-2}\wt{\vp}_{1}\sin^{n-3}\wt{\vp}_{2}\cdots\sin\wt{\vp}_{n-2} \D\wt{\vp}_{1}\cdots \D\wt{\vp}_{n-1}\\
&\quad=\frac{\rho^{n-2}r^{2}}{\left|r-A_{1}\right|}\sin^{n-2}\vp_{1}\cdots\sin\vp_{n-2} \D\vp_{1}\cdots \D\vp_{n-1}.
\end{align*}
This completes the proof.
\end{proof}
\subsection{Exterior problem}
In this section, we prove Theorem \ref{thm:int}.

We have
\[
g(\rho,\alpha)= \sum_{l={0}}^{\infty }g_{l}(\rho)\,Y_{l}(\alpha)
\]
and
\begin{equation}\label{Integral relation between g and f2}
\notag g(\rho,\alpha)=\sum_{l={0}}^{\infty }\int\limits_{S(\rho,\alpha)}f_{l}(r)\,Y_{l}(\varphi)\, \D \Omega.
\end{equation}
Using Proposition 3.3, we can write the above surface measure
$\D \O = \frac{\rho^{n-2}r^{2}}{\left|r-A_{1}\right|}\;\D \O(\vp)$, where $\D \O(\vp)= \sin^{n-2}\vp_{1}\cdots\sin\vp_{n-2} \D\vp_{1}\cdots \D\vp_{n-1}.$
Then
\begin{equation}\label{Integral relation between g and f}
\notag g(\rho,\alpha)= \sum_{l={0}}^{\infty}\int\limits_{\Sb^{n-1}}f_{l}(r)Y_{l}(\varphi) \frac{\rho^{n-2} r^{2}}{r-A_{1}}  \D \O(\vp).
\end{equation}

The integrand in \eqref{Integral relation between g and f} is to be interpreted as $0$ outside a suitable range of $\vp$. Now since $r= A_{1} + \sqrt{A_{1}^{2} +\rho^{2}-1}$, we have
\begin{align*}
g(\rho,\A)=\sum_{l={0}}^{\infty}&\int\limits_{\Sb^{n-1}}f_{l}(A+\sqrt{A^{2}+\rho^{2}-1}) \frac{\rho^{n-2} \lb A+\sqrt{A^{2}+\rho^{2}-1}\rb^{2}}{\sqrt{A^{2}+\rho^{2}-1}} \\
&\quad\quad\quad\times Y_{l}(\varphi)\D \O(\vp).
\end{align*}
Now we apply Funk-Hecke theorem.
\begin{theorem}
[Funk-Hecke]
If
\[ \int\limits_{-1}^{1} |F(t)|(1-t^{2})^{\frac{n-3}{2}} \D t <\infty,\] then
\[
\int\limits_{\Sb^{n-1}}F\lb \langle \sigma,\eta\rangle \rb Y_{l}(\sigma) \D \sigma = \frac{\left|\Sb^{n-2}\right|}{C_{l}^{\frac{n}{2}-1}(1)}\lb\int\limits_{-1}^{1} F(t) C_{l}^{\frac{n}{2}-1}(t) (1-t^{2})^{\frac{n-3}{2}} \D t\rb Y_{l}(\eta),
\]
where $\lvert \Sb^{n-2}\rvert$ denotes the surface measure of the unit sphere in $\Rb^{n-1}$ and $C_{l}^{\frac{n}{2}-1}$ are the Gegenbauer polynomials.
\end{theorem}

Using this theorem, we have,
\[
g_{l}(\rho)=\frac{\left|\Sb^{n-2}\right|}{C_{l}^{\frac{n}{2}-1}(1)}\int \limits_{1-\frac{\rho^{2}}{2}}^{1} f_{l}(x+\sqrt{x^{2}+\rho^{2}-1})  \frac{\rho^{n-2}\lb x+\sqrt{x^{2}+\rho^{2}-1}\rb^{2}}{\sqrt{x^{2}+\rho^{2}-1}} C_{l}^{\frac{n}{2}-1}(x)(1-x^{2})^{\frac{n-3}{2}} \D x .
\]

Making the change of variables $ r=x+\sqrt{x^2+\rho^2-1}$, we have
\begin{align*}
g_{l}(\rho)&=\frac{\rho^{n-2}\left|\Sb^{n-2}\right|}{C_{l}^{\frac{n}{2}-1}(1)}\int \limits_{1}^{1+\rho} f_{l}(r)r \lb C_{l}^{\frac{n}{2}-1}\lb\frac{r^{2}-\rho^{2}+1}{2r}\rb\rb\lb 1-\lb\frac{r^{2}-\rho^{2}+1}{2r}\rb^{2}\rb^{\frac{n-3}{2}} \D r\\
&=\frac{\rho^{n-2}\left|\Sb^{n-2}\right|}{C_{l}^{\frac{n}{2}-1}(1)}\int \limits_{0}^{\rho} f_{l}(r+1)(r+1) \lb C_{l}^{\frac{n}{2}-1}\lb\frac{r^{2}+2r-\rho^{2}+2}{2r+2}\rb\rb\lb 1-\lb\frac{r^{2}+2r-\rho^{2}+2}{2r+2}\rb^{2}\rb^{\frac{n-3}{2}} \D r
\end{align*}
This can be written in the form
\[
g_{l}(\rho)=\int\limits_{0}^{\rho} K_{l}(\rho,r) F_{l}(r) \D r,
\]
where
\begin{align*}
& K_{l}(\rho,r)= \frac{\rho^{n-2}\left|\Sb^{n-2}\right|}{C_{l}^{\frac{n}{2}-1}(1)}(r+1) \lb C_{l}^{\frac{n}{2}-1}\lb\frac{r^{2}+2r-\rho^{2}+2}{2r+2}\rb\rb\lb 1-\lb\frac{r^{2}+2r-\rho^{2}+2}{2r+2}\rb^{2}\rb^{\frac{n-3}{2}} \\
&F_{l}(r)=f_{l}(r+1).
\end{align*}
This is a Volterra integral equation of the first kind (see \cite{Volterra-book}). The kernel $K_{l}(\rho,r)$ is continuous together with its first derivatives and $K_{l}(\rho,\rho) \neq 0$ on the interval $(0,R_1)$, where $0<R_1<2$.  Equations of this type have a unique solution, which can be obtained through modification to a Volterra equation of the second kind,  and then using a resolvent kernel given by Picard's process of successive approximations (see \cite{Polyanin,Tricomi}). This completes the proof of Theorem \ref{thm:ext}.

\subsection{Interior problem}
Next we prove Theorem \ref{thm:int}.

Our starting point is:
\[
g(\rho,\A)=\sum_{l={0}}^{\infty }\int\limits_{\mathbb{S}^{n-1}}f_{l}(r) \frac{\rho^{n-2} r^{2}}{|r-A_{1}|} \,Y_{l}(\varphi)\D \O(\vp)
\]

We split the integral
\begin{align*}
\int\limits_{\mathbb{S}^{n-1}}f_{l}(r) \frac{\rho^{n-2} r^{2}}{|r-A_{1}|} Y_{l}(\varphi)\D \O(\vp)& =\int\limits_{\mathbb{S}^{n-1}_{+}}f_{l}(r) \frac{\rho^{n-2} r^{2}}{|r-A_{1}|} \,Y_{l}(\varphi) \D \O(\vp)\\
 &\quad \quad + \int\limits_{\mathbb{S}^{n-1}_{-}}f_{l}(r) \frac{\rho^{n-2} r^{2}}{|r-A_{1}|} \,Y_{l}(\varphi) \D \O(\vp),
\end{align*}
where $\Sb^{n-1}_{\pm}$ corresponds to those points on the unit sphere such that the line passing through it and the origin intersects a point on the sphere $S(\rho,\A)$ corresponding to $r=A_{1}\pm\sqrt{A_{1}^{2}+\rho^{2}-1}$.
Let us denote the right hand side of the above equation as $I_{1}+I_{2}$. We have
\[
I_{1}=\int\limits_{S^{n-1}_{+}} f_{l}\lb A_{1}+\sqrt{A_{1}^{2}+\rho^{2}-1}\rb \frac{\rho^{n-2} (A_{1}+\sqrt{A_{1}^{2}+\rho^{2}-1})^{2}}{\sqrt{A_{1}^{2}+\rho^{2}-1}}Y_{l}(\vp) \D \O(\vp).
\]
Applying Funk-Hecke theorem, this integral is
\[
I_{1}=\frac{\left|\Sb^{n-2}\right|}{C_{l}^{\frac{n}{2}-1}(1)}\lb\int \limits_{1-\frac{\rho^{2}}{2}}^{1} f_{l}(x+\sqrt{x^{2}+\rho^{2}-1})  \frac{\rho^{n-2}\lb x+\sqrt{x^{2}+\rho^{2}-1}\rb^{2}}{\sqrt{x^{2}+\rho^{2}-1}} C_{l}^{\frac{n}{2}-1}(x)(1-x^{2})^{\frac{n-3}{2}} \D x\rb Y_{l}(\A).
\]
Similarly
\[
I_{2}=\frac{\left|\Sb^{n-2}\right|}{C_{l}^{\frac{n}{2}-1}(1)}\lb \int \limits_{\sqrt{1-\rho^{2}}}^{1-\frac{\rho^{2}}{2}} f_{l}(x-\sqrt{x^{2}+\rho^{2}-1})  \frac{\rho^{n-2}\lb x-\sqrt{x^{2}+\rho^{2}-1}\rb^{2}}{\sqrt{x^{2}+\rho^{2}-1}} C_{l}^{\frac{n}{2}-1}(x)(1-x^{2})^{\frac{n-3}{2}} \D x \rb Y_{l}(\A).
\]
Making the change of variables $r=x+\sqrt{x^{2}+\rho^{2}-1}$ in $I_{1}$ and $r=x-\sqrt{x^{2}+\rho^{2}-1}$ and summing up the two integrals, we get,
\begin{align*}
g_{l}(\rho)&=\frac{\rho^{n-2}\left|\Sb^{n-2}\right|}{C_{l}^{\frac{n}{2}-1}(1)}\int \limits_{1-\rho}^{1} f_{l}(r)r \lb C_{l}^{\frac{n}{2}-1}\lb\frac{r^{2}-\rho^{2}+1}{2r}\rb\rb\lb 1-\lb\frac{r^{2}-\rho^{2}+1}{2r}\rb^{2}\rb^{\frac{n-3}{2}} \D r\\
&=\frac{\rho^{n-2}\left|\Sb^{n-2}\right|}{C_{l}^{\frac{n}{2}-1}(1)}\int \limits_{0}^{\rho} f_{l}(1-r)(1-r) \lb C_{l}^{\frac{n}{2}-1}\lb\frac{r^{2}-2r-\rho^{2}+2}{2-2r}\rb\rb\lb 1-\lb\frac{r^{2}-2r-\rho^{2}+2}{2-2r}\rb^{2}\rb^{\frac{n-3}{2}} \D r
\end{align*}
This is of the form
\[
g_{l}(\rho)=\int\limits_{0}^{\rho} K_{l}(\rho,r) F_{l}(r) \D r,
\]
where
\begin{align*}
& K_{l}(\rho,r)= \frac{\rho^{n-2}\left|\Sb^{n-2}\right|}{C_{l}^{\frac{n}{2}-1}(1)}(1-r) \lb C_{l}^{\frac{n}{2}-1}\lb \frac{r^{2}-2r-\rho^{2}+2}{2-2r}\rb\rb\lb 1-\lb\frac{r^{2}-2r-\rho^{2}+2}{2-2r}\rb^{2}\rb^{\frac{n-3}{2}} \\
&F_{l}(r)=f_{l}(1-r).
\end{align*}
Note that $K_n(\rho,\rho)$ does not vanishes in the interval $(0,1-\varepsilon)$ and its derivatives exist and are continuous. The rest of the proof follows exactly as in Theorem \ref{thm:ext}.

\subsection{Interior/exterior problem}
Finally we prove Theorem \ref{thm:int-ext}.

Since the argument is exactly as in Theorems \ref{thm:ext} and \ref{thm:int}, we will only give the final integral identity. Assume that the function $f$ is supported inside the ball  $B(0,R_{2})$ centered at the origin and of radius $R_{2}$, where $R_{2}>2$ and $R_{1}=R_{2}-2$. Suppose the spherical Radon transform data is known along all spheres of radius $\rho$ centered on the unit sphere  with $R_{2}-1<\rho< R_{2}+1$, then we have the following Volterra-type integral equation:
\begin{align*}
g_{l}(\rho)&=\frac{\rho^{n-2}\left|\Sb^{n-2}\right|}{C_{l}^{\frac{n}{2}-1}(1)}\int \limits_{\rho-1}^{R_{2}} f_{l}(r)r \lb C_{l}^{\frac{n}{2}-1}\lb\frac{r^{2}-\rho^{2}+1}{2r}\rb\rb\lb 1-\lb\frac{r^{2}-\rho^{2}+1}{2r}\rb^{2}\rb^{\frac{n-3}{2}} \D r\\
&=\frac{\rho^{n-2}\left|\Sb^{n-2}\right|}{C_{l}^{\frac{n}{2}-1}(1)}\int \limits_{0}^{R_2+1-\rho} f_{l}(R_{2}-r)(R_{2}-r) \lb C_{l}^{\frac{n}{2}-1}\lb\frac{(R_{2}-r)^{2}-\rho^{2}+1}{2(R_{2}-r)}\rb\rb\\
&\hspace{1.5in} \times \lb 1-\lb\frac{(R_{2}-r)^{2}-\rho^{2}+1}{2(R_2-r)}\rb^{2}\rb^{\frac{n-3}{2}} \D r
\end{align*}
Making a change of variable $\hat{\rho}=R_2+1-\rho$ we get
\[
G_{l}(\hat{\rho})=\int\limits_{0}^{\hat{\rho}} K_{l}(\hat{\rho},r) F_{l}(r) \D r,
\]
where
\begin{align*}
 &K_{l}(\hat{\rho},r)= \frac{{(R_2+1-\hat{\rho})}^{n-2}\left|\Sb^{n-2}\right|}{C_{l}^{\frac{n}{2}-1}(1)}(R_{2}-r)  \lb C_{l}^{\frac{n}{2}-1}\lb\frac{(R_{2}-r)^{2}-{(R_2+1-\hat{\rho})}^{2}+1}{2(R_{2}-r)}\rb\rb\\
&\hspace{1in} \times\lb 1-\lb\frac{(R_{2}-r)^{2}-{(R_2+1-\hat{\rho})}^{2}+1}{2(R_2-r)}\rb^{2}\rb^{\frac{n-3}{2}} \\
&F_{l}(r)=f_{l}(R_{2}-r), \;\;G_{l}(\hat{\rho})=g_{l}(R_2+1-\hat{\rho}).
\end{align*}
The rest of the proof follows exactly as before.

\section{Three dimensional case}\label{3d}
In the numerical simulations below, we specialize to the case of $3$-dimensions.  Therefore, in this section, we give the formulas derived earlier for the case of $n=3$.

In this section, for the sake of convenience, we rename the vector $\A$ as $(\A,\B)$ and the vector $\vp$ as $(\vp,\theta)$. Thus in this section and the next, the point $C$ will be denoted by $(\A,\B)$, more precisely, the Euclidean coordinates of the point $C$ on the unit sphere will be denoted by $(\cos \A,\sin\A\cos \B,\sin\A\sin\B)$. A point $P$ on the sphere $S(\rho,\A,\B)$ will be denoted by $(r\cos \vp, r\sin\vp \cos \theta,r\sin\vp \sin\theta)$.

Here the spherical harmonics for $f$ and $Rf=g$ are expanded as
\begin{equation}\label{3fg-expansion}
f(r,\vp,\theta)=\sum_{l={0}}^{\infty }\sum_{m=-l}^{l}f_{l}^{m}(r)\,Y_{l}^{m}(\vp,\theta).
\end{equation}
\begin{equation}\label{3g-expansion}
g(\rho,\alpha,\beta)=\sum_{l={0}}^{\infty }\sum_{m=-l}^{l}g_{l}^{m}(\rho)\,Y_{l}^{m}(\alpha,\beta).
\end{equation}

In the case of $3$-dimensions, we have that $C_{l}^{(\frac{1}{2})}(x)=P_{l}(x)$, where $P_{l}(x)$ are the Legendre polynomials and $C_{l}^{(\frac{1}{2})}(1)=1$. Therefore the relation between the spherical harmonics coefficients in the three cases are as follows:
\begin{enumerate}
\item[] (Exterior case)
\begin{align}\label{3d-exterior-case}
\notag & g_{l}^{m}(\rho)=\int_{0}^{\rho}F_{l}^{m}(r)\,K_{l}(\rho,r) \D r,\\
& K_{l}(\rho,r)=2\pi\rho(r+1)\lb P_l\left(\frac{r^2-\rho^2+2r+2}{2(r+1)}\right)\rb\\
\notag &F_{l}^{m}(r)=f_{l}^{m}(r+1).
\end{align}
\item[] (Interior case)
\begin{align}\label{3d-interior-case}
\notag & g_{l}^{m}(\rho)=\int_{0}^{\rho}F_{l}^{m}(r) K_{l}(\rho,r) \D r,\\
& K_{l}(\rho,r)=2\pi\rho(1-r)\lb P_l \lb \frac{r^{2} -\rho^{2} + 2-2r}{2(1-r)}\rb\rb\\
\notag & F_{l}^{m}(r)=f_{l}^{m}(1-r).
\end{align}
\item[] {(Interior/exterior case)}
\begin{align}\label{3d-intext-case}
\notag & G_{l}^{m}(\rho)=\int_{0}^{\rho}F_{l}^{m}(r) K_{l}(\rho,r) \D r,\\
& K_{l}(\rho,r)=2\pi(R_{2}+1-\rho)(R_{2}-r)\lb  P_l \lb \frac{(R_{2}-r)^{2}+1 -{(R_{2}+1-\rho)}^{2}}{2(R_{2}-r)}\rb\rb\\
\notag & F_{l}^{m}(r)=f_{l}^{m}(R_{2}-r),\;\; G_l^m(\rho)=g_l^m(R_{2}+1-\rho).
\end{align}

Note, that in all three cases the kernel $K_{l}(\rho,r)$ is bounded and has a continuous first derivative on the support of (corresponding) $F_l^m$,
and $K_l(\rho,\rho)\ne 0$. Hence, these Volterra equations of the first kind can be transformed into equations of the second kind
and solved using a resolvent kernel given by Picard's process of successive approximations (see \cite{Polyanin,Tricomi}).

\end{enumerate}

\section{Numerical Algorithm}\label{numerical_algorithms}
\subsection{Generating the Radon data}
We consider a generic sphere of integration $S(\rho,\alpha,\beta)$ to be centered at $C=(a_1,b_1,c_1)$ and radius $\rho$ where the center $(a_1,b_1,c_1)$ lies on the sphere of radius $R$. For the interior and exterior cases, we choose $R=1$ and thus use the formulas \eqref{3d-exterior-case} and \eqref{3d-interior-case} derived in the previous sections. For the combined interior and exterior case, we use $R=1.49$ and note that \eqref{3d-intext-case} can be easily generalized for acquisition spheres of radius $R$. We consider test phantoms $f$ to be disjoint unions of characteristic functions of balls. To find the spherical Radon transform of $f$, we need to find the surface area of intersection of $S(\rho,\alpha,\beta)$ with $f$. This is equivalent to summing up the surface area of intersection of $S(\rho,\alpha,\beta)$ with characteristic function of each ball. Thus, in the forthcoming calculations, we consider a ball $B$ centered at $(a_2,b_2,c_2)$ and radius $a$.

The sphere $S(\rho,\A,\B)$ and the ball $B$ intersect only when the following conditions do not occur:
$$
\sqrt{(a_1-a_2)^2+(b_1-b_2)^2+(c_1-c_2)^2} > (\rho+a)
$$
and
$$
\sqrt{(a_1-a_2)^2+(b_1-b_2)^2+(c_1-c_2)^2} <\rho - a.
$$
To compute the surface area of intersection of $S(\rho,\alpha,\beta)$ with $B$ in these cases, we first determine the center of the circle of intersection of $S$ and $B$, denoted by $(x_c,y_c,z_c)$. The equation of the plane $\mathcal{P}$ passing through the intersection $S$ and $B$ is given as follows
$$
(a_2-a_1)x+(b_2-b_1)y+(c_2-c_1)y = \dfrac{\rho^2-a^2+a_2^2+b_2^2+c_2^2-a_1^2-b_1^2-c_1^2}{2}.
$$
The equation of the straight line passing through $(x_c,y_c,z_c)$ and perpendicular to the plane $\mathcal{P}$ is given as follows
$$
\dfrac{x-a_1}{a_2-a_1}=\dfrac{y-b_1}{b_2-b_1}=\dfrac{z-c_1}{c_2-c_1}=t,\qquad t\in \mathbb{R}.
$$
We then can compute
\begin{align*}
& x_c = \dfrac{(\rho^2-a^2)(a_2-a_1)}{2Z}+\dfrac{a_1+a_2}{2}, \\
& y_c = \dfrac{(\rho^2-a^2)(b_2-b_1)}{2Z}+\dfrac{b_1+b_2}{2},\\
& z_c = \dfrac{(\rho^2-a^2)(c_2-c_1)}{2Z}+\dfrac{c_1+c_2}{2},
\end{align*}
where
$Z= (a_1-a_2)^2+(b_1-b_2)^2+(c_1-c_2)^2.$
Let $d$ be the distance between $(x_c,y_c,z_c)$ from the center of $S$. Then, by an elementary calculation, the surface area of intersection of $S(\rho,\alpha,\beta)$ with $B$, denoted by $\mathcal{S}$, is given as
$$
\mathcal{S} = 2\pi\rho(\rho-d).
$$
Thus
$$
\mathcal{S}=2\pi\rho^2-2\pi\rho\cdot\frac{|\rho^2-a^2+(a_1-a_2)^2+(b_1-b_2)^2+(c_1-c_2)^2|}{2\sqrt{(a_1-a_2)^2+(b_1-b_2)^2+(c_1-c_2)^2}}.
$$

\subsection{Evaluating the spherical harmonics coefficients of the Radon data}
After obtaining the Radon data $g(\rho,\alpha,\beta)$, we need to determine $g_{l}^{m}(\rho),~ l=0,\hdots,\infty,$ $~m=-l,\hdots,l$ given by
$$
g_{l}^{m}(\rho)= \int_0^{2\pi}\int_0^\pi g(\rho,\alpha,\beta) \bar{Y}^m_l(\alpha,\beta)\D \alpha \D \beta
$$
where $\bar{Y}^m_l(\alpha,\beta) = (-1)^{-m}Y^{-m}_l(\alpha,\beta).$
This is done numerically by following the method employed in \cite{blais}. Given $g(\rho,\alpha,\beta)$ 	at $\alpha_j = \pi j/2N,\beta_k=\pi k/N, j,k=0,\hdots, 2N-1$, we compute
$$
g_l^m(\rho) = \dfrac{1}{N} \sqrt{\dfrac{\pi}{2}}\sum_{j=0}^{2N-1}\sum_{k=0}^{2N-1}a_j g(\rho,\alpha_j,\beta_k)\bar{Y}^m_l(\alpha_j,\beta_k),
$$
where
$$
a_j = \dfrac{\sqrt{2}}{N}\sin\left({\dfrac{\pi j}{2N}}\right)\sum_{p=0}^{N-1}\dfrac{1}{2p+1}\sin\left({(2p+1)\dfrac{\pi j}{2N}}\right).
$$
\subsection{Inversion of the integral equations {\bf \eqref{3d-exterior-case}}, \eqref{3d-interior-case} and \eqref{3d-intext-case}}
To solve the integral equations numerically, we use the product trapezoidal method as found in \cite{PlatoArticle, RKCV, Weiss_Product_Integration_Paper}. In this method, the integral equations are discretized and the integrands are approximated by product trapezoidal rule. This in turn leads to matrix-vector equations and thus, we obtain  discrete solutions of the discretized integral equations, provided the matrices are invertible. In the following section, we provide the corresponding matrix-vector equations for solving \eqref{3d-exterior-case}, \eqref{3d-interior-case} and \eqref{3d-intext-case} corresponding to the exterior, interior and the combined interior/exterior problems respectively and prove their invertibility.
\subsubsection{Exterior case}
We rewrite \eqref{3d-exterior-case} as
$$
g_{l}^{m}(\rho)=\int_{0}^{\rho}f_{l}^{m}(r+1)(1+r)\wt{K}_l(\rho,r) \D r,
$$
where
\[
\wt{K}(\rho,r)=2\pi\rho P_l\left(\frac{r^2-\rho^2+2r+2}{2r+2}\right),~ 0<\rho <1.
\]
We discretize $\rho \in (0,1)$ into $M+1$ equidistant points of interval length $h$ as $\rho_i,~ i = 0,\hdots, M$. The corresponding matrix-vector equation is given as follows

\begin{equation}\label{Mat_ext}
A_E\vec{f}_l^m = \vec{g}_{l}^{m}
\end{equation}
where
$$
\vec{f}_l^m = \left(
\begin{array}{c}
f_l^m(1+\rho_0)\\
\vdots\\
f_l^m(1+\rho_M)\\
\end{array}
\right),~ \vec{g}_l^m = \left(
\begin{array}{c}
g_l^m(\rho_0)\\
\vdots\\
g_l^m(\rho_M)\\
\end{array}
\right)
$$
and $A_E = (a_{ik})$ where
$$
a_{ik} =
\begin{cases}
\wt{K}_l(\rho_i,\rho_0)\left[{\frac{h(\rho_{1}+\rho_{0})}{6}+\frac{h\rho_{0}}{6}+\frac{h}{2}}\right],~ k=0\\
\wt{K}_l(\rho_i,\rho_k)\left[{\frac{h(\rho_{k-1}+4\rho_{k}+\rho_{k+1})}{6}+h}\right], ~1\leq k\leq i-1\\
\wt{K}_l(\rho_i,\rho_i)\left[{\frac{h(\rho_{i}+\rho_{i-1})}{6}+\frac{h\rho_{i}}{6}+\frac{h}{2}}\right],~ k =i\\
0,~ k>i.
\end{cases}
$$
Note that $A_{E}$ is a lower triangular matrix, and since
\begin{equation}
\begin{aligned}
a_{00}=&\wt{K}_l(\rho_0,\rho_0)\left[{\frac{h(\rho_{1}+\rho_{0})}{6}+\frac{h\rho_{0}}{6}+\frac{h}{2}}\right]= 2\pi\rho_0 \left[{\frac{h(\rho_{1}+2 \rho_{0}+3)}{6}}\right] >0, ~ \mbox{ if } \rho_0 >0.\\
a_{ii}=&\wt{K}_l(\rho_i,\rho_i)\left[{\frac{h(\rho_{i}+\rho_{i-1})}{6}+\frac{h\rho_{i}}{6}+\frac{h}{2}}\right]
= 2\pi\rho_i \left[{\frac{h(2\rho_{i}+\rho_{i-1}+3)}{6}}\right] >0,
\end{aligned}
\end{equation}
we have that $A_E$ is invertible.

\subsubsection{Interior Case}
We again discretize $\rho \in (0,1)$ into $M+1$ equidistant points as $\rho_i,~ i = 0,\hdots, M$ and obtain the following matrix-vector equation
\begin{equation}\label{Mat_int}
 A_I\vec{f}_l^m = \vec{g}_l^m
\end{equation}
where
$$
\vec{f}_l^m = \left(
\begin{array}{c}
f_l^m(1-\rho_0)\\
\vdots\\
f_l^m(1-\rho_M)\\
\end{array}
\right),~ \vec{g}_l^m = \left(
\begin{array}{c}
g_l^m(\rho_0)\\
\vdots\\
g_l^m(\rho_M)\\
\end{array}
\right)
$$
and $A_I = (a_{ik})$ where
$$
a_{ik} =
\begin{cases}
\wt{K}_l(\rho_i,\rho_0)\left[{\frac{-h(\rho_{1}+\rho_{0})}{6}+\frac{-h\rho_{0}}{6}+\frac{h}{2}}\right],~ k=0\\
\wt{K}_l(\rho_i,\rho_k)\left[{\frac{-h(\rho_{k-1}+4\rho_{k}+\rho_{k+1})}{6}+h}\right], ~1\leq k\leq i-1\\
\wt{K}_l(\rho_i,\rho_i)\left[{\frac{-h(\rho_{i}+\rho_{i-1})}{6}+\frac{-h\rho_{i}}{6}+\frac{h}{2}}\right],~ k =i\\
0,~ k>i.
\end{cases}
$$
Therefore, if $\rho_0 >0$
\begin{equation}
\begin{aligned}
a_{00}=&\wt{K}_l(\rho_0,\rho_0)\left[{\frac{-h(\rho_{1}+\rho_{0})}{6}+\frac{-h\rho_{0}}{6}+\frac{h}{2}}\right]
= 2\pi\rho_0 \left[{\frac{-h(\rho_{1}+2 \rho_{0}-3)}{6}}\right] >0.\\
a_{ii}=&\wt{K}_l(\rho_i,\rho_i)\left[{\frac{-h(\rho_{i}+\rho_{i-1})}{6}+\frac{-h\rho_{i}}{6}+\frac{h}{2}}\right]
= 2\pi\rho_i \left[{\frac{-h(2\rho_{i}+\rho_{i-1}-3)}{6}}\right] >0.
\end{aligned}
\end{equation}
Thus $A_I$ is invertible.

\subsubsection{Interior/Exterior Case}
In a similar way as in the previous two cases, we discretize $\rho \in (0,2R)$ into $M+1$ equidistant points as $\rho_i, ~i = 0,\hdots, M$ to obtain the following matrix-vector equation
\begin{equation}\label{Mat_intext}
 A_{IE}\vec{f}_l^m = \vec{g}_l^m
\end{equation}
where
$$
\vec{f}_l^m = \left(
\begin{array}{c}
f_l^m(R_2-\rho_0)\\
\vdots\\
f_l^m(R_2-\rho_M)\\
\end{array}
\right),~ \vec{g}_l^m = \left(
\begin{array}{c}
g_l^m(\rho_0)\\
\vdots\\
g_l^m(\rho_M)\\
\end{array}
\right)
$$
and $A_{IE} = (a_{ik})$ where
$$
a_{ik} =
\begin{cases}
\wt{K}_l(\rho_i,\rho_0)\left[{\frac{-h(\rho_{1}+\rho_{0})}{6}+\frac{-h\rho_{0}}{6}+\frac{hR_2}{2}}\right],~ k=0\\
\wt{K}_l(\rho_i,\rho_k)\left[{\frac{-h(\rho_{k-1}+4\rho_{k}+\rho_{k+1})}{6}+hR_2}\right], ~1\leq k\leq i-1\\
\wt{K}_l(\rho_i,\rho_i)\left[{\frac{-h(\rho_{i}+\rho_{i-1})}{6}+\frac{-h\rho_{i}}{6}+\frac{hR_2}{2}}\right],~ k =i\\
0,~ k>i.
\end{cases}
$$
Therefore, if $\rho_0 >0$
\begin{equation}
\begin{aligned}
a_{00}=&\wt{K}_l(\rho_0,\rho_0)\left[{\frac{-h(\rho_{1}+\rho_{0})}{6}+\frac{-h\rho_{0}}{6}+\frac{hR_2}{2}}\right]
= 2\pi\rho_0 \left[{\frac{-h(\rho_{1}+2 \rho_{0}-3R_2)}{6}}\right] >0.\\
a_{ii}=&\wt{K}_l(\rho_i,\rho_i)\left[{\frac{-h(\rho_{i}+\rho_{i-1})}{6}+\frac{-h\rho_{i}}{6}+\frac{hR_2}{2}}\right]
= 2\pi\rho_i \left[{\frac{-h(2\rho_{i}+\rho_{i-1}-3R_2)}{6}}\right] >0.
\end{aligned}
\end{equation}
Thus $A_{IE}$ is invertible.

The following theorem states the error estimate for the numerical solution of the integral equations \eqref{3d-exterior-case}, \eqref{3d-interior-case} and \eqref{3d-intext-case} which follows from \cite[Thm. 7.2]{Linz-Book}.
\begin{theorem}[Error Estimates]\label{Error Estimates}
Let $f_{l}^{m,\mathrm{exact}}$ be the $C^{3}$ solution of \eqref{3d-exterior-case} [\eqref{3d-interior-case}, \eqref{3d-intext-case} in $[0,R]$ and $f_{l}^m$ be the solution to \ref{Mat_ext} [or (\ref{Mat_int}) and (\ref{Mat_intext}) resp.]. Then
\Beq
 \Vert f_{l}^{m,\mathrm{exact}} - f_{l}^m\rVert_2= \Oc(h^2),
\Eeq
where $\Vert\cdot\rVert_2$ represents the discrete version of the continuous $L^2$ norm in $[0,R]$ (see for e.g., \cite[Ch. 4]{book_L2norm}).
\end{theorem}

To solve the matrix equations (\ref{Mat_ext}), (\ref{Mat_int}) and (\ref{Mat_intext}), we need to
invert the matrices $A_E, A_I, A_{IE}$. It turns out that the condition numbers of these matrices are greater than $10^4$ for almost all values of $l,m$. It is well known that numerically inverting a matrix with condition number$10^r$ leads to a loss of
$r$ digits of accuracy \cite{Hansen-TSVD}. Thus for inversion, we use the technique of Truncated Singular
Value Decomposition (TSVD), originally proposed in \cite{Golub-Kahan}. See also \cite{AR,RKCV}.

\section{Numerical Results}\label{num_res}
In this section we show the results of the numerical computations performed for the inversion of spherical transforms described in Section \ref{main_results} with functions supported in interior, exterior and both interior and exterior of the acquisition sphere. We discretize $\rho\in[\epsilon,R-\epsilon]$, with $\epsilon=0.001$, into 50 equally spaced grid points, $\alpha,\theta\in [0,\pi]$ and $\beta,\phi\in [0,2\pi]$ into 100 equally spaced grid points for all our computations. As mentioned before, for the interior and the exterior cases, the value of $R=1$ whereas for the combined interior and exterior case, the value of $R=1.49$.

\subsection{Functions supported inside the acquisition sphere}
\begin{figure}[h]
  \centering
    \subfloat[]{%
     \label{interior_phantom_exact_hor} \includegraphics[scale=0.33,keepaspectratio]{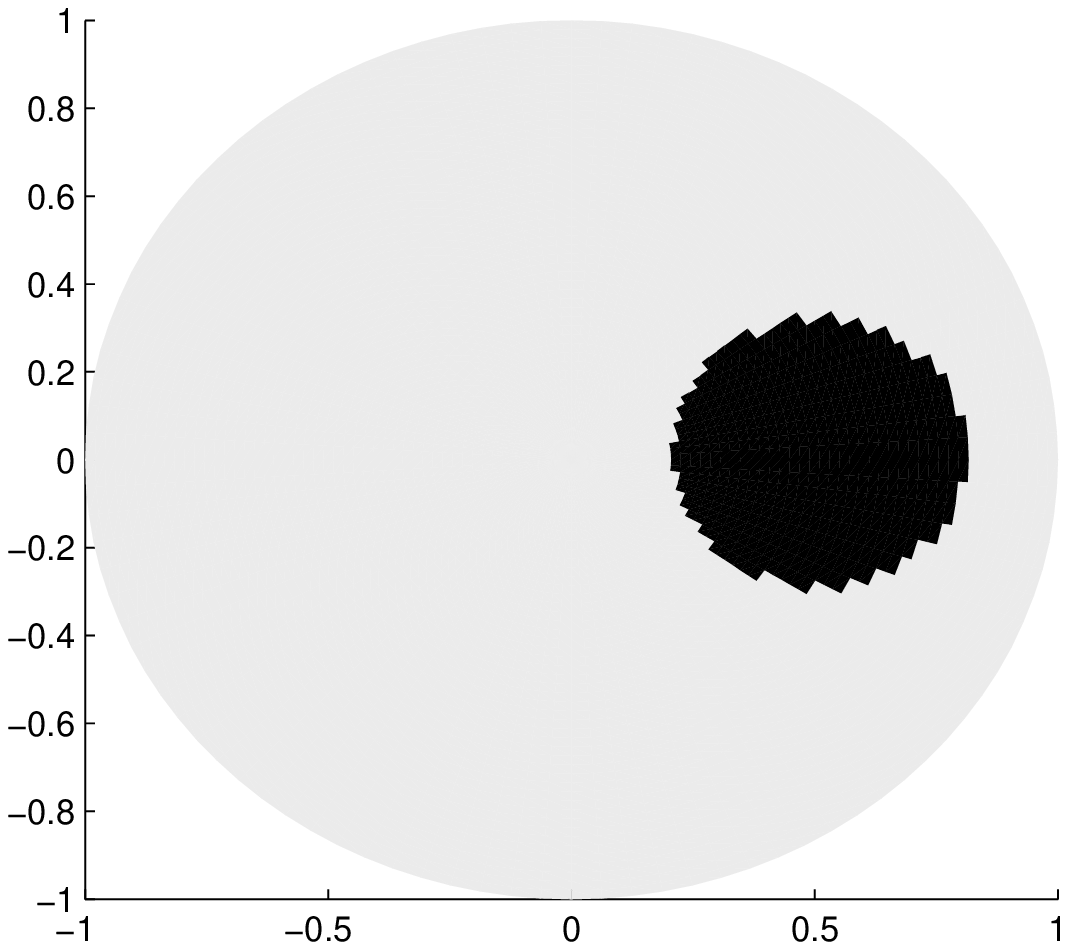}
   }
    \subfloat[]{%
      \label{interior_phantom_exact_ver}\includegraphics[scale=0.33,keepaspectratio]{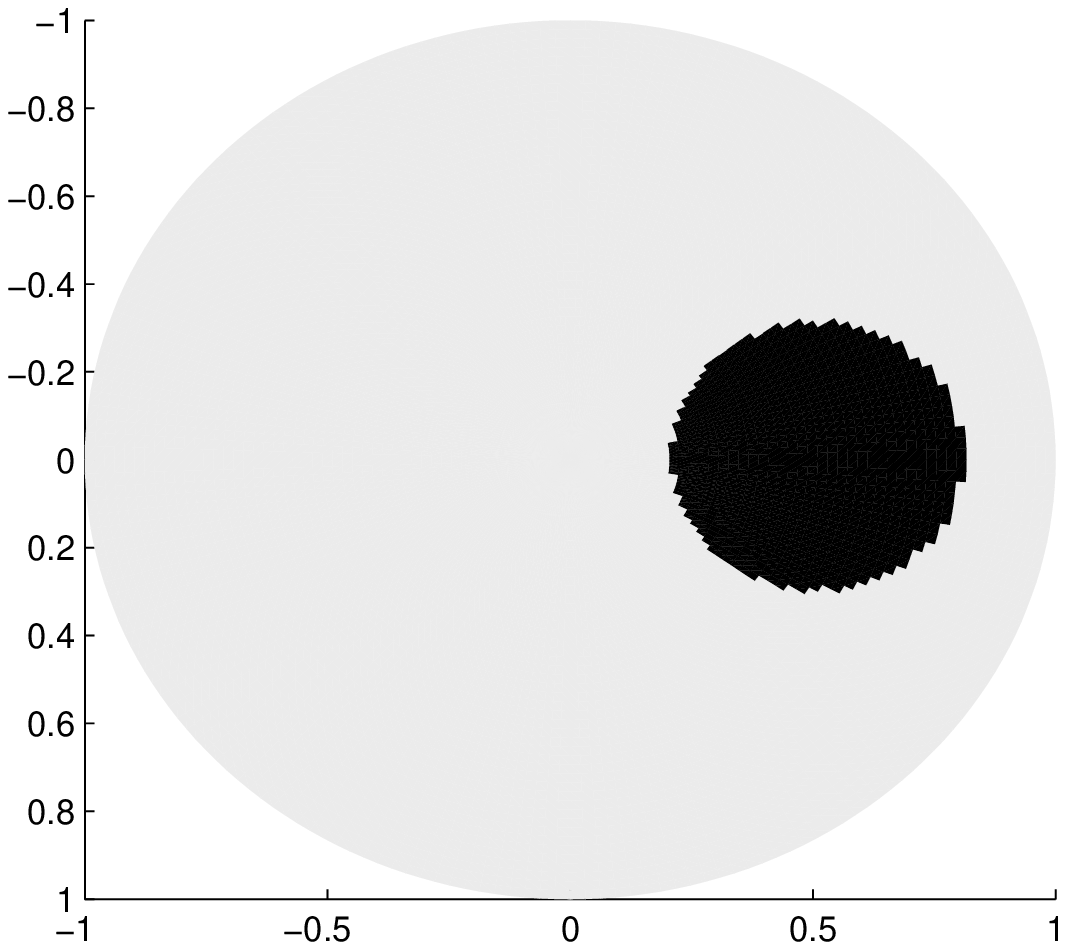}
      }\\
      \subfloat[]{%
      \label{Reconstruction_interior_hor}\includegraphics[scale=0.33,keepaspectratio]{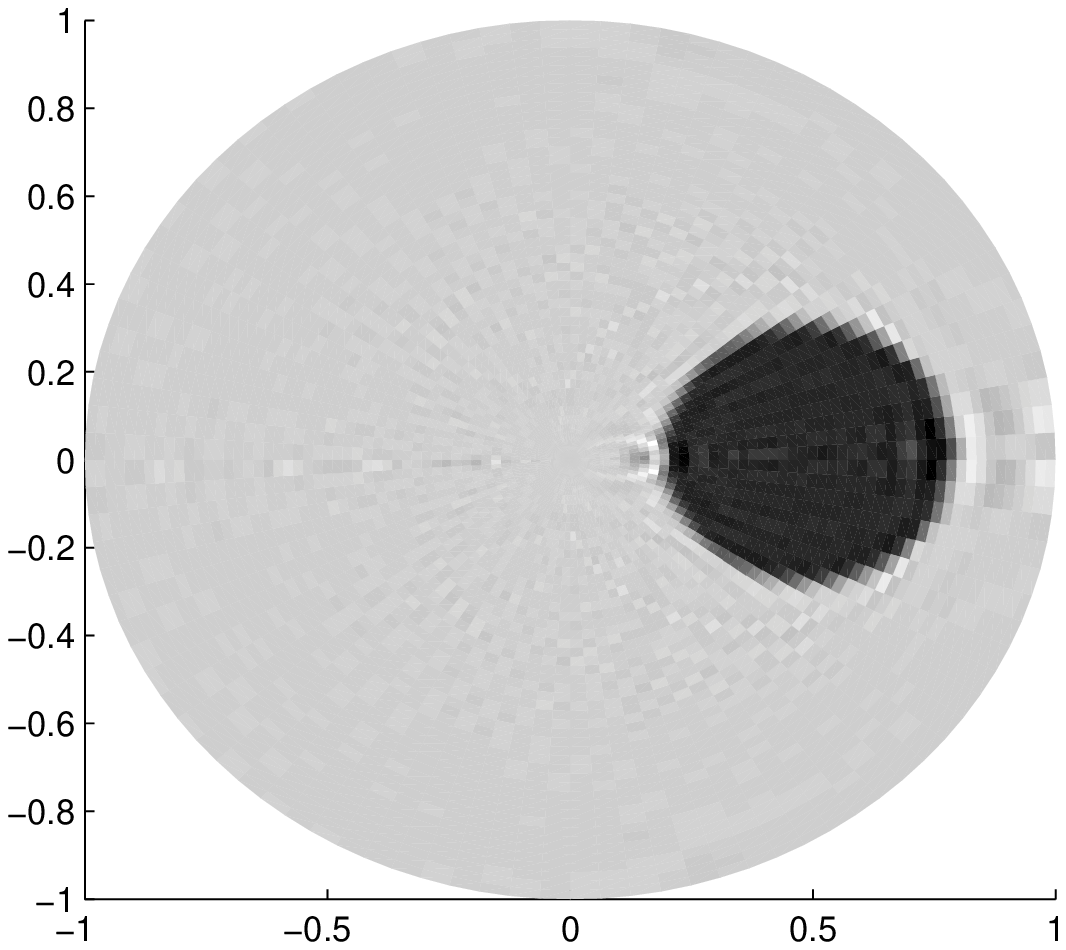}
      }
            \subfloat[]{%
      \label{Reconstruction_interior_ver}\includegraphics[scale=0.33,keepaspectratio]{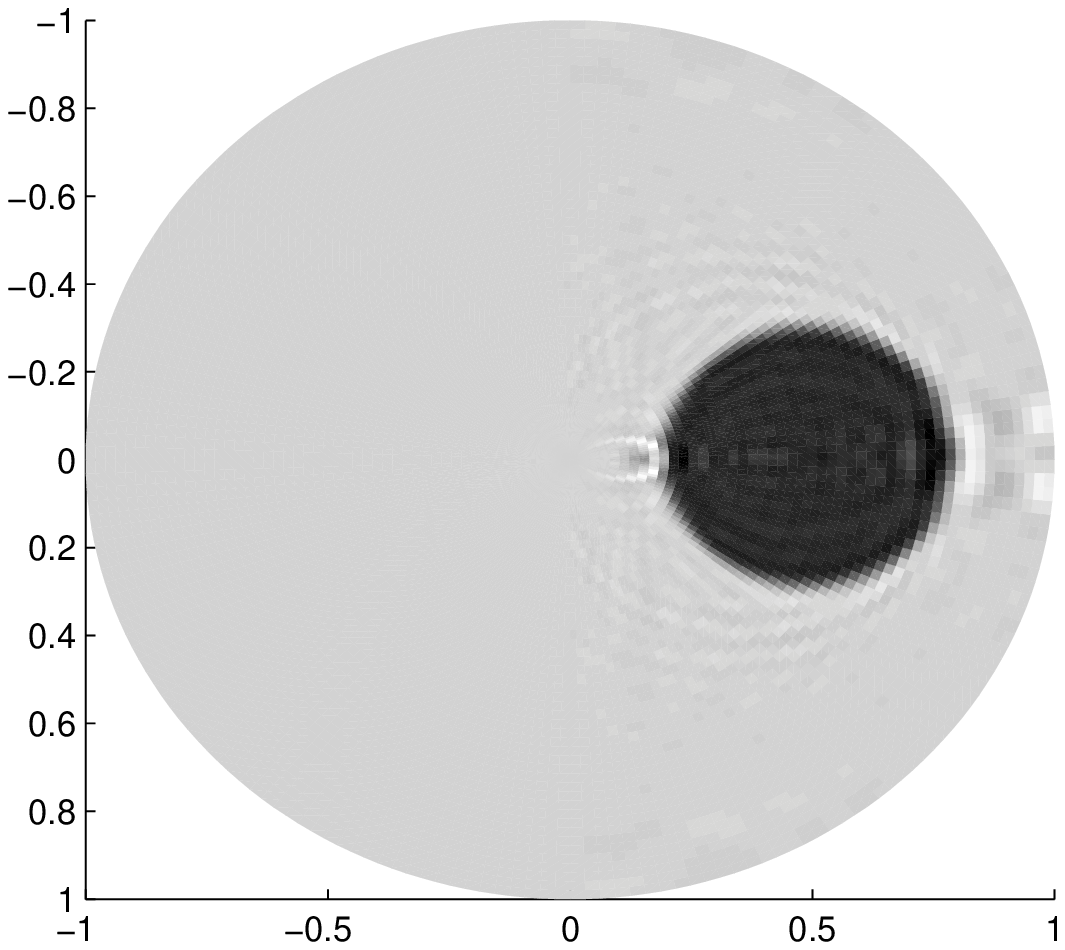}
}
      \caption{Results for spherical Radon transform data for a function supported inside the acquisition sphere. Figures \ref{interior_phantom_exact_hor} and \ref{interior_phantom_exact_ver} represent the horizontal and the vertical views of the actual phantom. Figures \ref{Reconstruction_interior_hor} and \ref{Reconstruction_interior_ver} show the horizontal and vertical views of the reconstructed images.}
      \label{fig:interior}
      \end{figure}

Figures \ref{interior_phantom_exact_hor} and \ref{interior_phantom_exact_ver} show the horizontal and vertical cross sections of a phantom represented by a ball centered at $(0.5,0,0)$ and radius 0.3. Figure \ref{Reconstruction_interior_hor} and \ref{Reconstruction_interior_ver} shows the horizontal and the vertical cross sections of the reconstructed phantom. We note the good recovery in this case.

To demonstrate the robustness of our algorithm, we also tested it on the spherical Radon data with 5\% multiplicative Gaussian noise. The results are shown in Figures \ref{Reconstruction_interior_noise_hor} and  \ref{Reconstruction_interior_noise_ver}. We again note the good recovery in presence of noisy data.

\begin{figure}[h]
  \centering

     \subfloat[]{%
      \label{Reconstruction_interior_noise_hor}\includegraphics[scale=0.33,keepaspectratio]{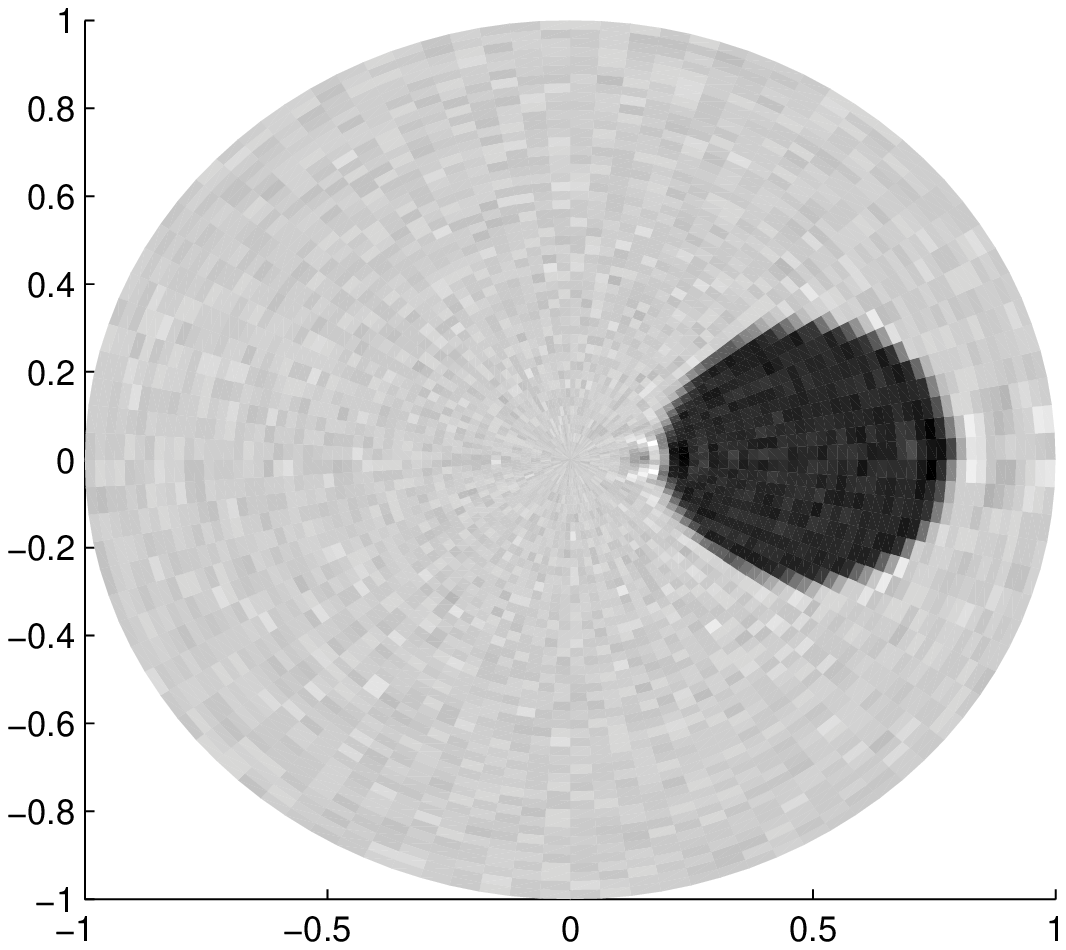}
      }
            \subfloat[]{%
      \label{Reconstruction_interior_noise_ver}\includegraphics[scale=0.33,keepaspectratio]{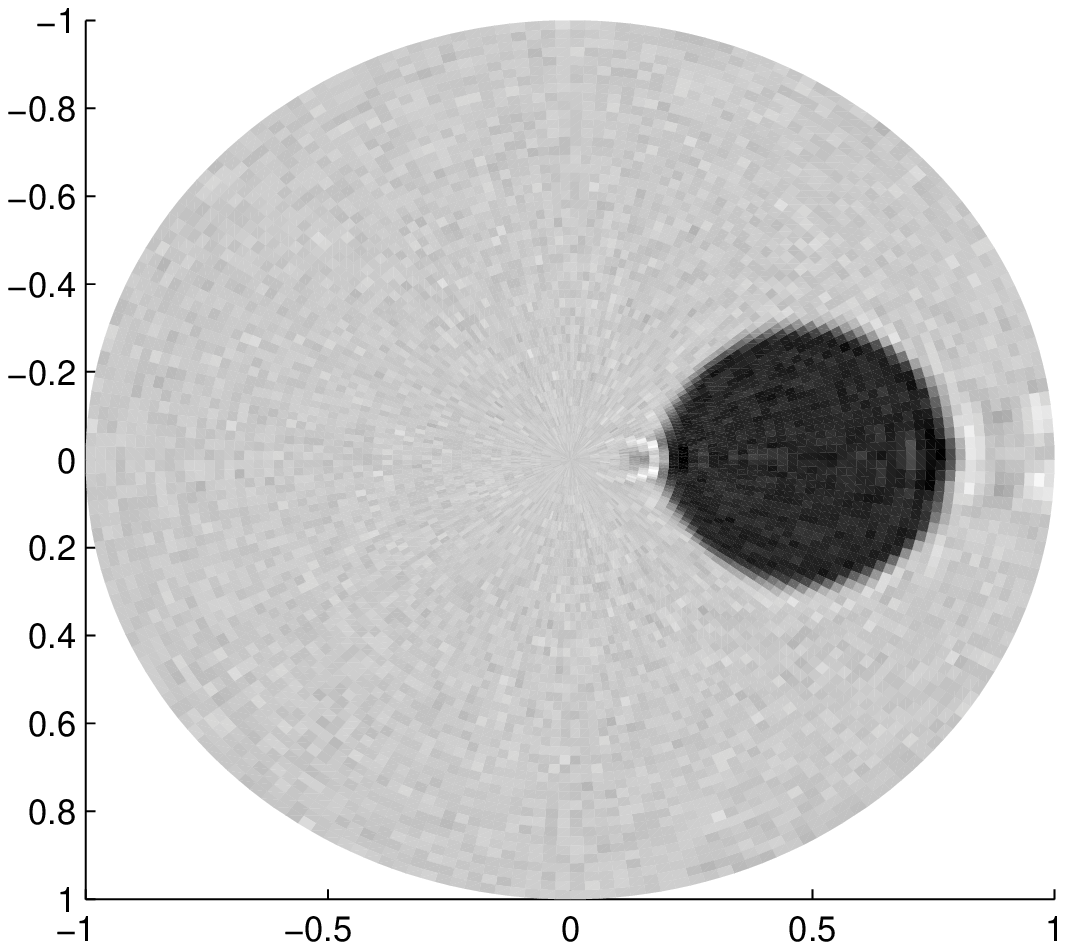}
}
      \caption{Results for spherical Radon transform data with 5\% multiplicative Gaussian noise for a function supported inside the acquisition sphere. Figures \ref{Reconstruction_interior_noise_hor} and \ref{Reconstruction_interior_noise_ver} show the horizontal and vertical views of the reconstructed images.}
      \label{fig:interior_noise}
      \end{figure}

We also applied our reconstruction algorithm to a phantom whose support is inside the acquisition sphere and contains the origin. Notice, that our result about uniqueness of the inversion (Thm. \ref{thm:int}) does not cover this case, since here the kernels $K_l(\rho,r)$ of the integral equations appearing in the proof vanish, when $\rho=r=1$ (see eq. (\ref{3d-interior-case})). Hence, one may not expect stable recovery in the numerical method. And indeed, the reconstructed image in Figure \ref{Reconstruction_interior_origin_hor} shows instability near the origin.
\begin{figure}[h]
  \centering
    \subfloat[]{%
     \label{interior_phantom_exact_origin_hor} \includegraphics[scale=0.33,keepaspectratio]{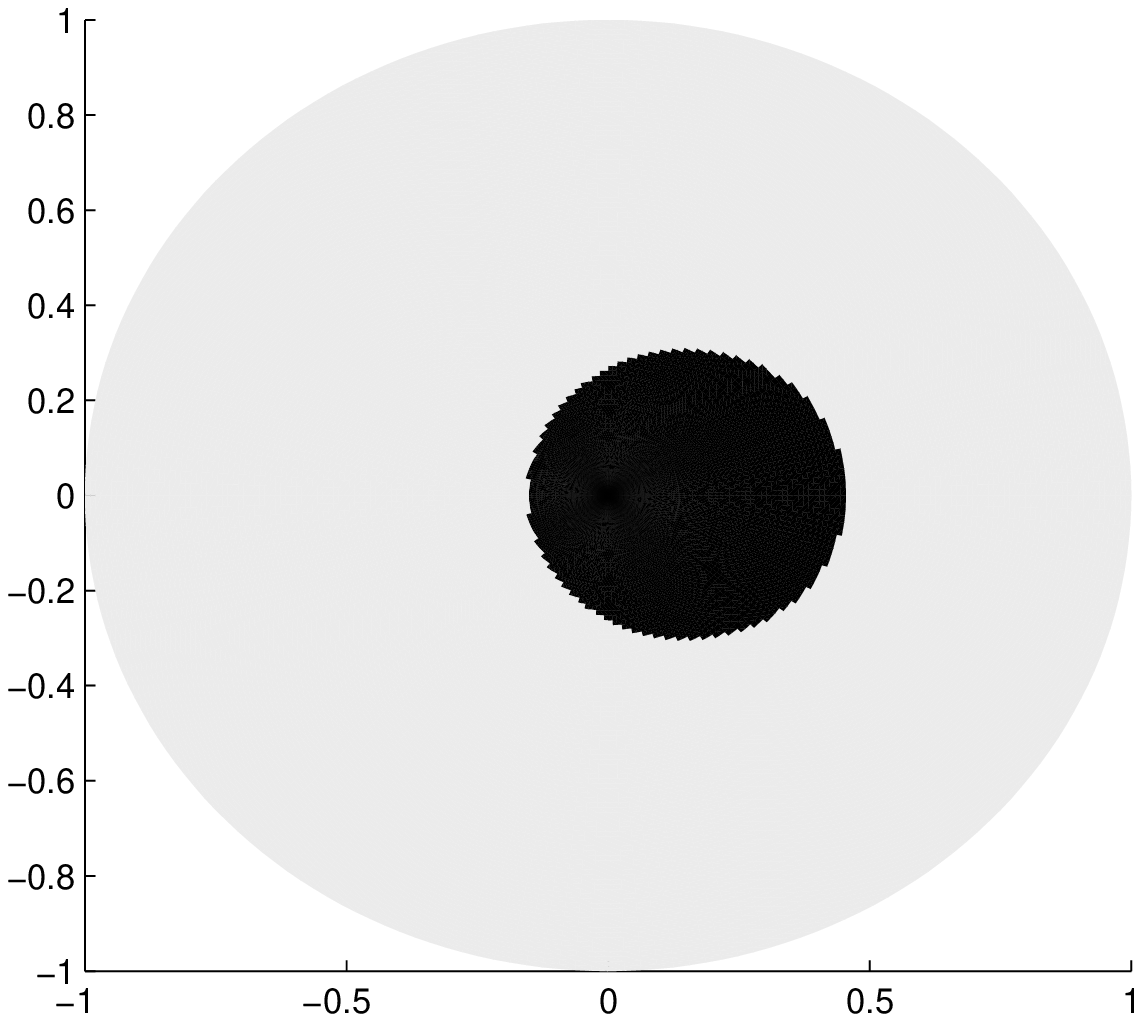}
   }
     \subfloat[]{%
      \label{Reconstruction_interior_origin_hor}\includegraphics[scale=0.33,keepaspectratio]{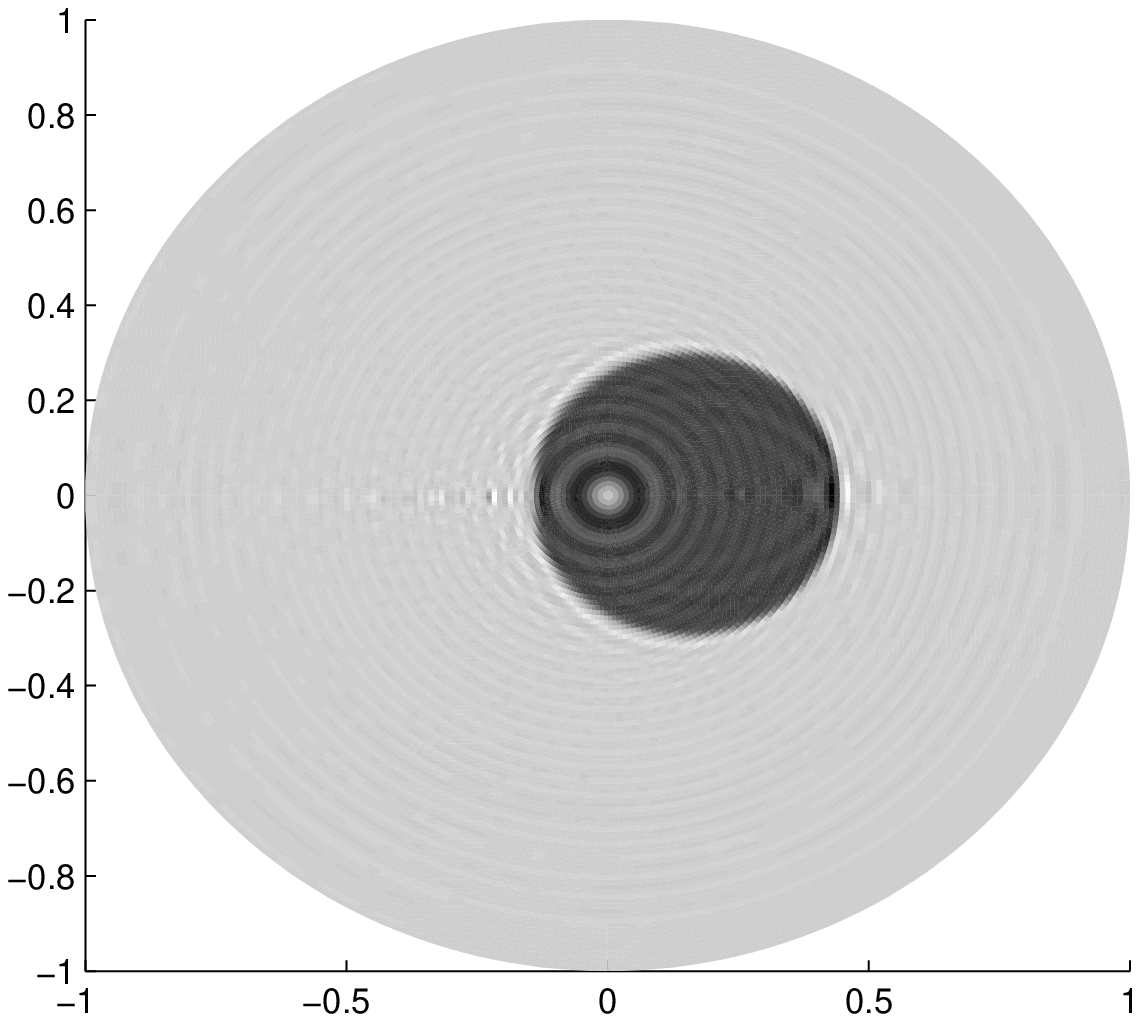}
      }
      \subfloat[]{%
      \label{Reconstruction_interior_origin_cross_section}\includegraphics[scale=0.33,keepaspectratio]{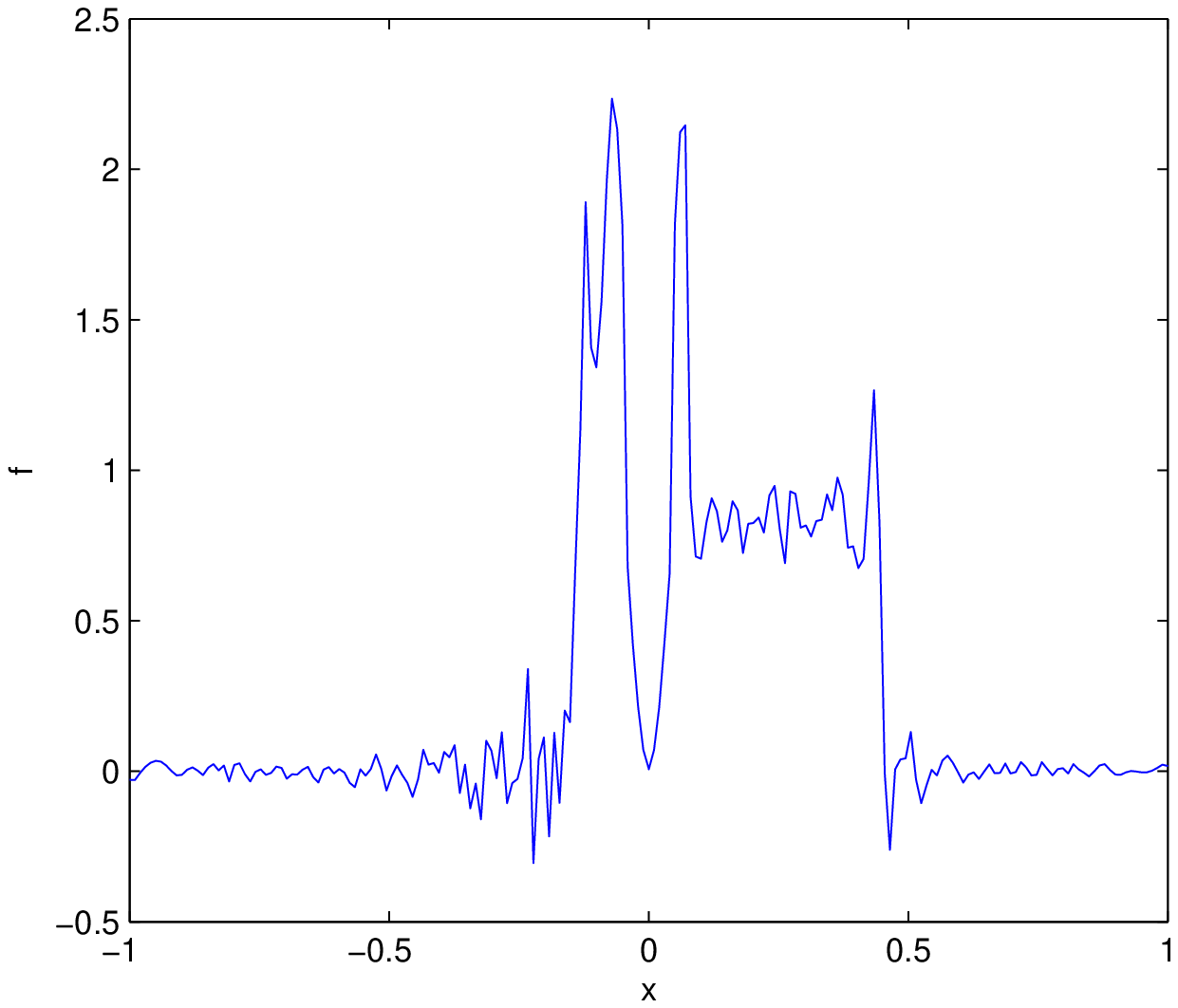}
      }

      \caption{Application of the algorithm to a function supported inside the acquisition sphere with support containing the origin. Figure \ref{interior_phantom_exact_origin_hor} shows the horizontal view of the actual phantom. Figure \ref{Reconstruction_interior_origin_hor} shows the horizontal view of the reconstructed image. Figure \ref{Reconstruction_interior_origin_cross_section} shows the cross sectional view of the reconstructed phantom along the x-axis. }
      \label{fig:interior_origin}
      \end{figure}
\subsection{Functions supported outside the acquisition sphere}
\begin{figure}[h]
  \centering
    \subfloat[]{%
     \label{exterior_phantom_exact_hor} \includegraphics[scale=0.33,keepaspectratio]{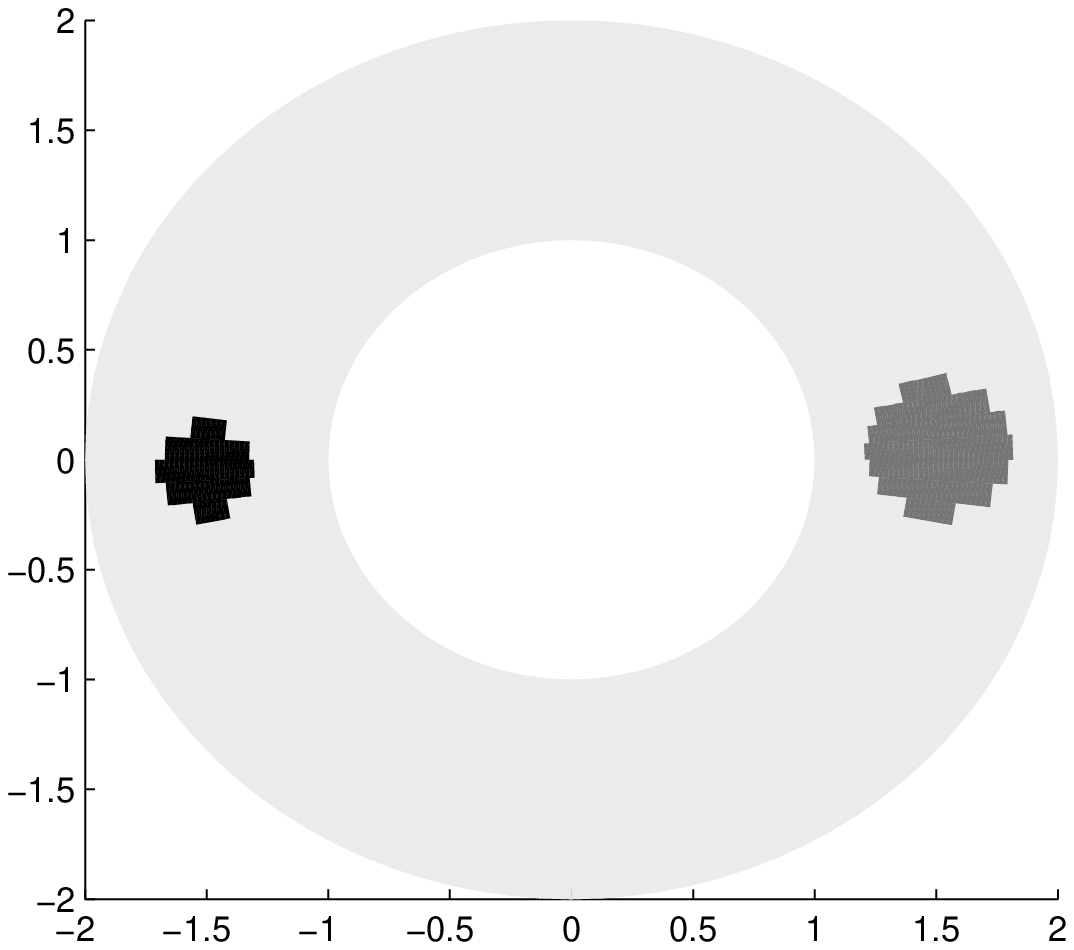}
   }
    \subfloat[]{%
      \label{exterior_phantom_exact_ver}\includegraphics[scale=0.33,keepaspectratio]{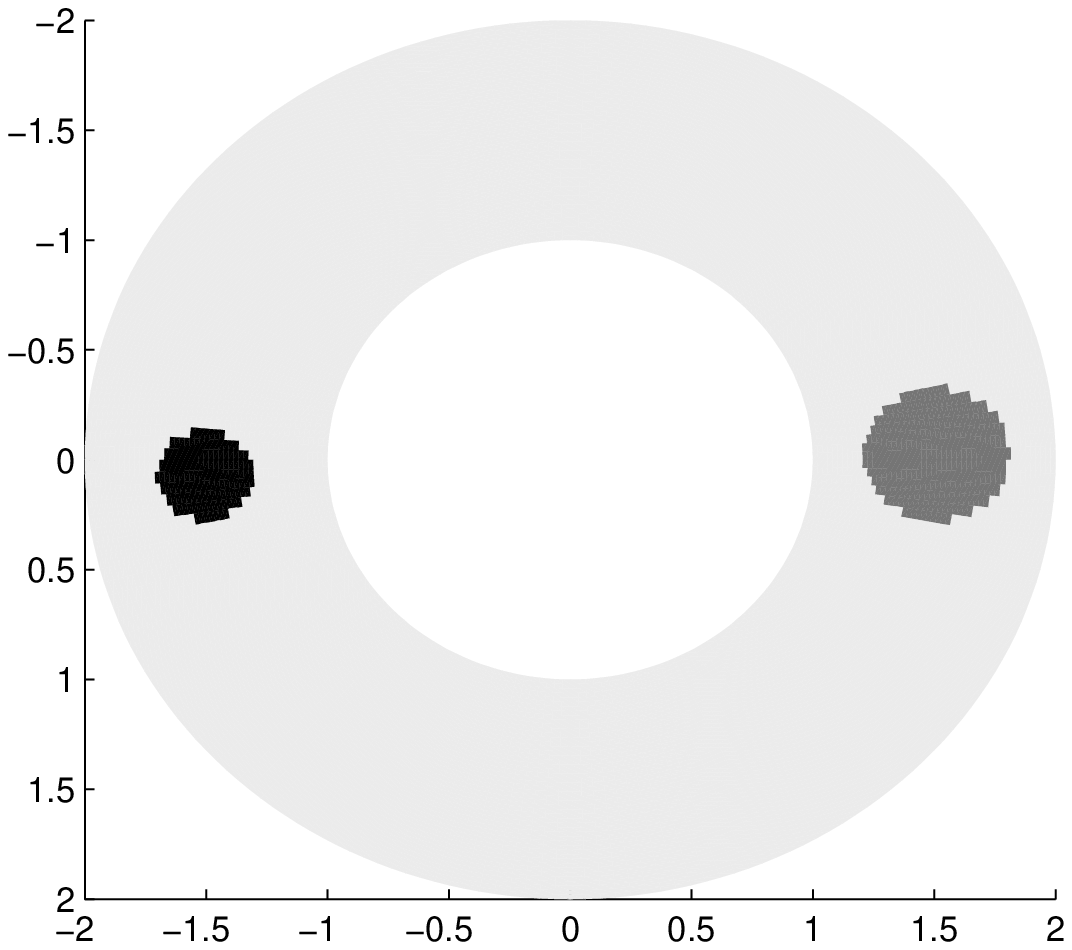}
      }\\
      \subfloat[]{%
      \label{Reconstruction_exterior_hor}\includegraphics[scale=0.33,keepaspectratio]{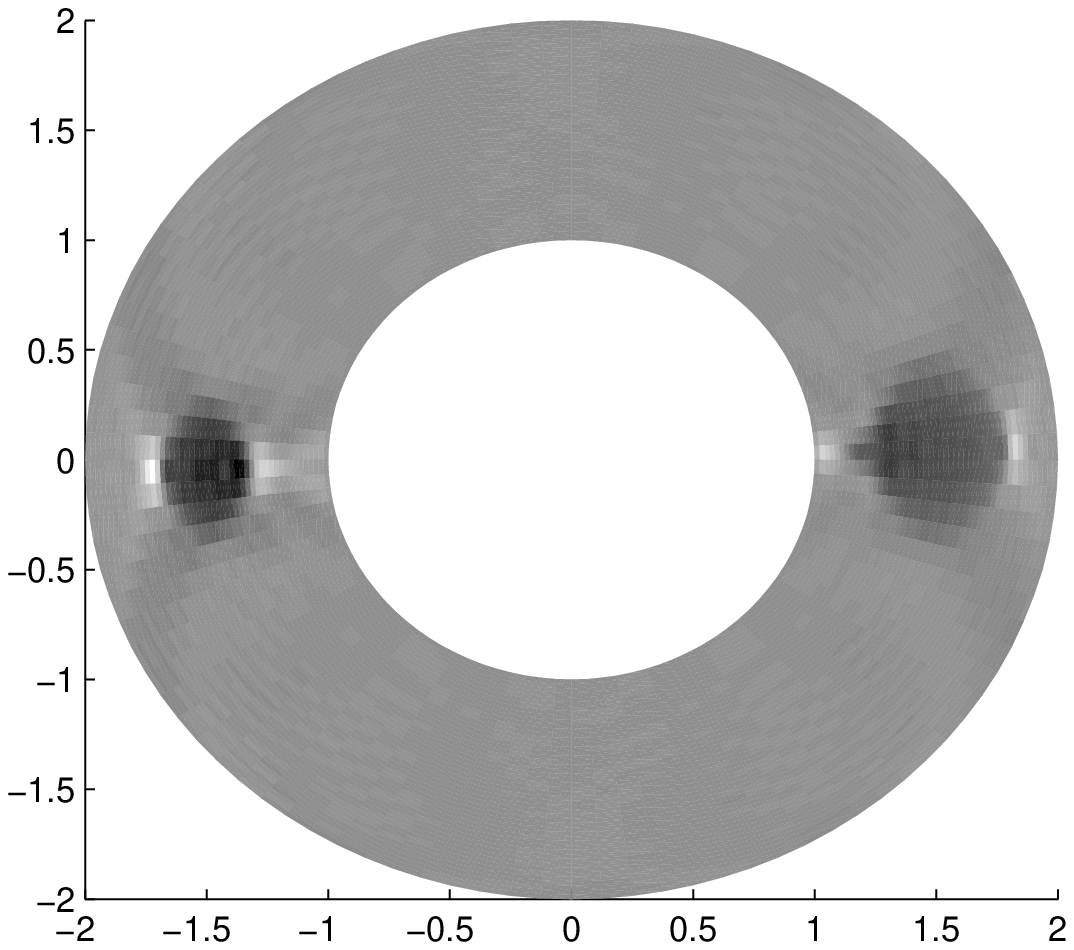}
      }
            \subfloat[]{%
      \label{Reconstruction_exterior_ver}\includegraphics[scale=0.33,keepaspectratio]{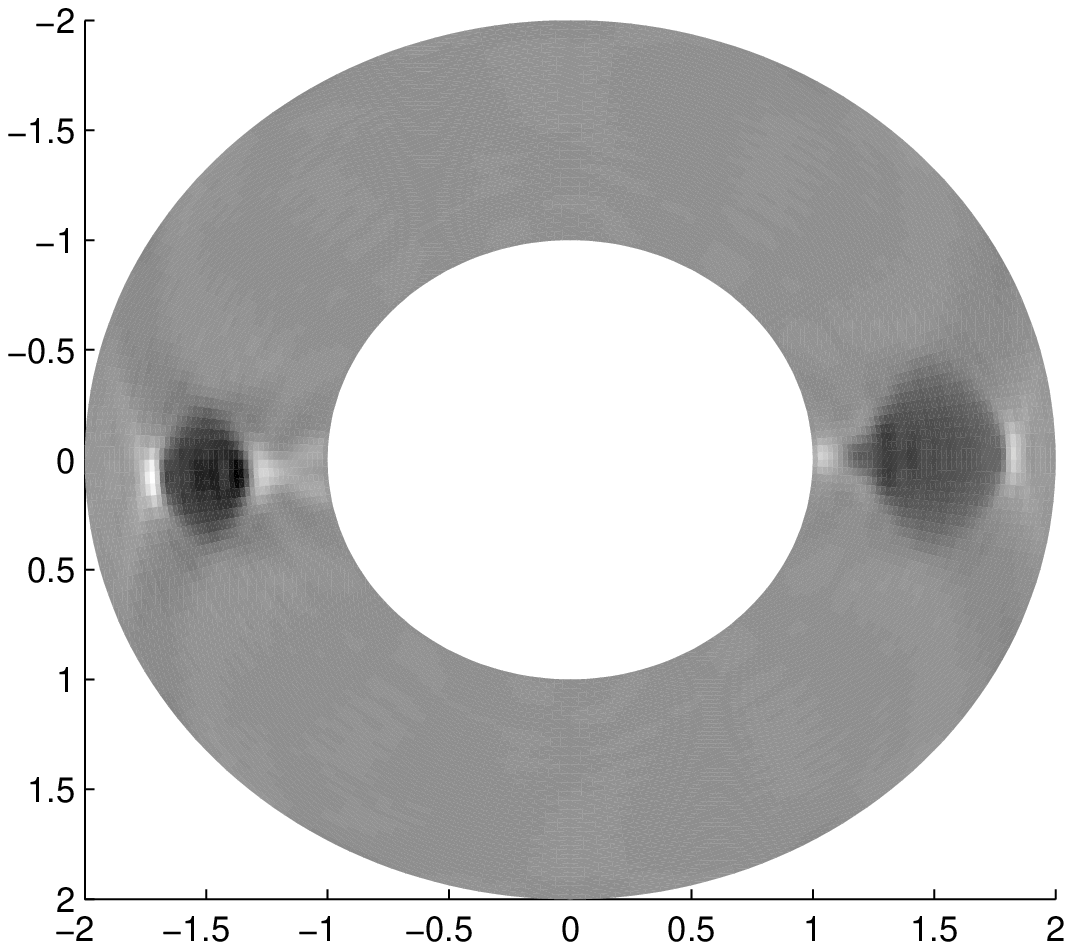}
}
      \caption{Results for spherical Radon transform data for a function supported outside the acquisition sphere. Figures \ref{exterior_phantom_exact_hor} and \ref{exterior_phantom_exact_ver} represent the horizontal and the vertical views of the actual phantom. Figures \ref{Reconstruction_exterior_hor} and \ref{Reconstruction_exterior_ver} show the horizontal and vertical views of the reconstructed images.}
      \label{fig:exterior1}
      \end{figure}

Figures \ref{exterior_phantom_exact_hor} and \ref{exterior_phantom_exact_ver} show the horizontal and vertical cross sections of a phantom represented by two balls centered at $(-1.5,0,0)$ and $(1.5,0,0)$ with radius 0.2 and 0.3 respectively. Figures \ref{Reconstruction_exterior_hor} and \ref{Reconstruction_exterior_ver} shows the horizontal and the vertical cross sections of the reconstructed phantom. Microlocal analysis arguments show that the entire spherical shell of the balls cannot be constructed stably with the given spherical Radon transform data.
We see the presence of an increased number of artifacts in contrast to the interior case. The reconstructions are consistent with this analysis.

\subsection{Functions supported on both sides of the acquisition sphere}

\begin{figure}[h]
  \centering
    \subfloat[]{%
     \label{interior_exterior_phantom_exact_hor} \includegraphics[scale=0.33,keepaspectratio]{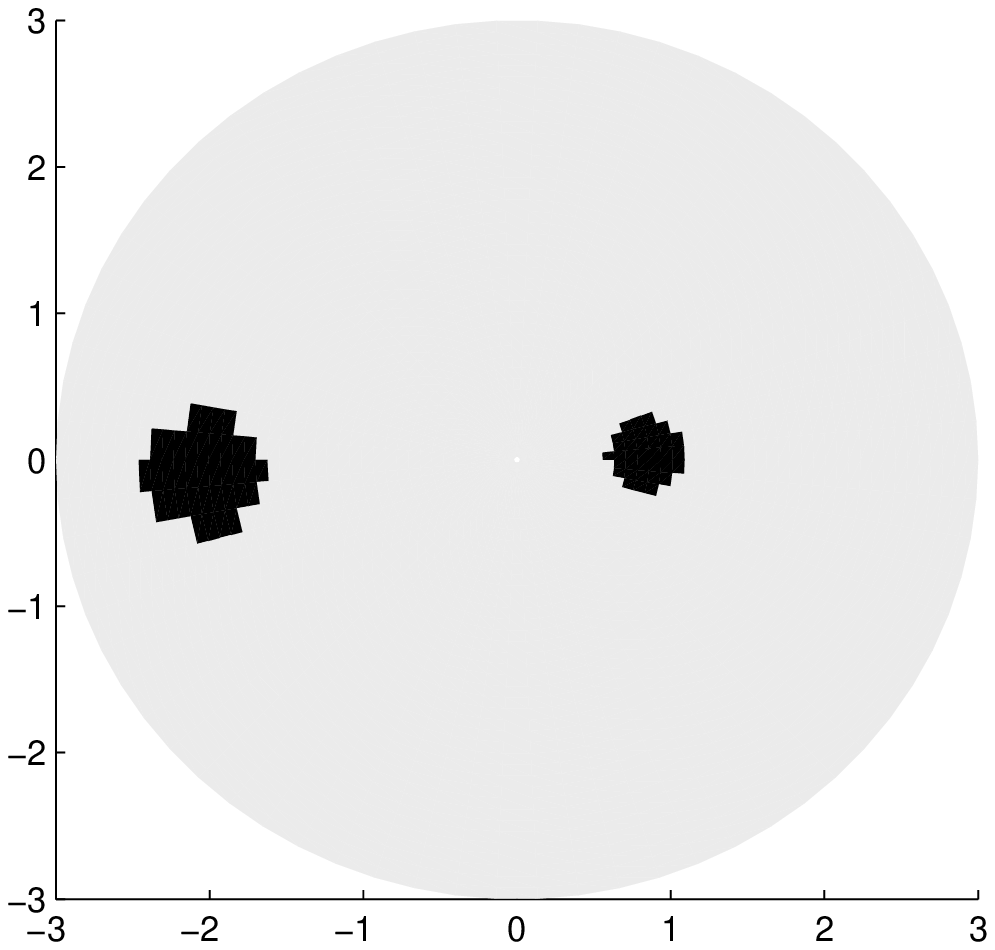}
   }
    \subfloat[]{%
      \label{Reconstruction_interior_exterior_hor}\includegraphics[scale=0.33,keepaspectratio]{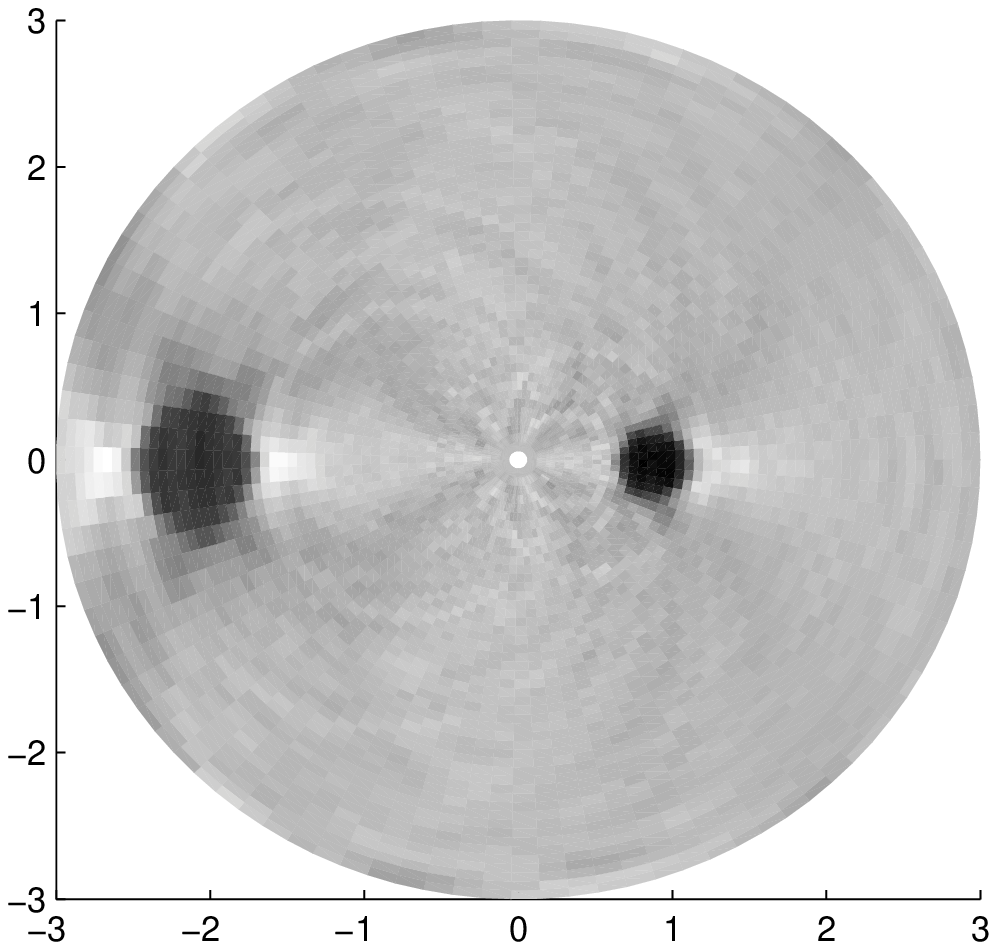}
      }\\

      \caption{Results for spherical Radon transform data for a function supported on both sides of the acquisition sphere. Figure \ref{interior_exterior_phantom_exact_hor} represents the horizontal view of the actual phantom. Figure \ref{Reconstruction_interior_exterior_hor} shows the horizontal view of the reconstructed image.}
      \label{fig:exterior2}
      \end{figure}

Figure \ref{interior_exterior_phantom_exact_hor} shows the horizontal cross section of a phantom represented by two balls centered at $(0.5,0,0)$ and $(-2.0,0,0)$ with radius 0.2 and 0.3 respectively. The balls lie on either side of the  acquisition sphere, which is centered at the origin with radius 1.49. Figure \ref{Reconstruction_interior_exterior_hor} shows the horizontal cross section of the reconstructed phantoms. Again by microlocal analysis arguments, the ball outside the acquisition sphere cannot be constructed stably whereas the ball inside the acquisition sphere can be constructed stably. This is depicted in the reconstructions.

\section{Conclusion}

We studied the problem of inverting the spherical Radon transform in spherical geometry of data acquisition with incomplete radial data. Such problems arise in image reconstruction procedures in photo- and thermo-acoustic tomography, ultrasound reflection tomography, as well as in radar and sonar imaging. We considered three distinct scenarios of the location of the support of the image function: strictly inside the acquisition sphere (interior problem), strictly outside (exterior problem), and both inside and outside (interior/exterior problem). For all three cases we provided a constructive proof of the uniqueness of inversion of SRT from incomplete radial data and obtained an iterative procedure to recover the image function. We presented a robust computational algorithm based on our inversion procedure and demonstrated its accuracy and efficiency on several numerical examples.

\section*{Acknowledgments}

Ambartsoumian was supported in part by US NSF Grants DMS 1109417, DMS 1616564 and Simons Foundation Grant 360357.

Gouia-Zarrad was supported in part by the American University of Sharjah (AUS) research grant FRG3.

Krishnan was supported in part by NSF grants DMS 1109417 and DMS 1616564. He and Roy benefited from support of the Airbus Group Corporate Foundation Chair ``Mathematics of Complex Systems'' established at TIFR Centre for Applicable Mathematics
and TIFR International Centre for Theoretical Sciences, Bangalore, India.

\begin{thebibliography}{10}
\bibitem{ABK-Spherical-Lp}
Mark Agranovsky, Carlos Berenstein, and Peter Kuchment.
\newblock Approximation by spherical waves in {$L^p$}-spaces.
\newblock {\em J. Geom. Anal.}, 6(3):365--383 (1997), 1996.

\bibitem{AKQ-Spherical}
Mark Agranovsky, Peter Kuchment, and Eric~Todd Quinto.
\newblock Range descriptions for the spherical mean {R}adon transform.
\newblock {\em J. Funct. Anal.}, 248(2):344--386, 2007.

\bibitem{Agranovsky-Quinto}
Mark~L. Agranovsky and Eric~Todd Quinto.
\newblock Injectivity sets for the {R}adon transform over circles and complete
  systems of radial functions.
\newblock {\em J. Funct. Anal.}, 139(2):383--414, 1996.

\bibitem{Agranovsky-Quinto-Duke}
Mark~L. Agranovsky and Eric~Todd Quinto.
\newblock Geometry of stationary sets for the wave equation in {$\Bbb R^n$}:
  the case of finitely supported initial data.
\newblock {\em Duke Math. J.}, 107(1):57--84, 2001.

\bibitem{AGL-Circular}
Gaik Ambartsoumian, Rim Gouia-Zarrad, and Matthew~A. Lewis.
\newblock Inversion of the circular {R}adon transform on an annulus.
\newblock {\em Inverse Problems}, 26(10):105015, 11, 2010.

\bibitem{Ambartsoumian-Krishnan}
Gaik Ambartsoumian and Venkateswaran~P. Krishnan.
\newblock Inversion of a class of circular and elliptical {R}adon transforms.
\newblock In {\em Complex analysis and dynamical systems {VI}. {P}art 1},
  volume 653 of {\em Contemp. Math.}, pages 1--12. Amer. Math. Soc.,
  Providence, RI, 2015.

\bibitem{AK1}
Gaik Ambartsoumian and Peter Kuchment.
\newblock On the injectivity of the circular {R}adon transform.
\newblock {\em Inverse Problems}, 21(2):473--485, 2005.

\bibitem{AK2}
Gaik Ambartsoumian and Peter Kuchment.
\newblock A range description for the planar circular {R}adon transform.
\newblock {\em SIAM J. Math. Anal.}, 38(2):681--692, 2006.

\bibitem{AR}
Gaik Ambartsoumian and Souvik Roy.
\newblock Numerical inversion of a broken ray transform arising in single
  scattering optical tomography.
\newblock {\em IEEE Trans. Comput. Imaging}, 2(2):166--173, 2016.

\bibitem{AZSZP}
Mark~A. Anastasio, Jin Zhang, Emil~Y. Sidky, Yu~Zou, Dan Xia, and Xiaochuan
  Pan.
\newblock Feasibility of half-data image reconstruction in 3-d reflectivity
  tomography with a spherical aperture.
\newblock {\em IEEE Transactions on Medical Imaging}, 24(9):1100--1112, Sept
  2005.

\bibitem{Andersson_1988}
Lars-Erik Andersson.
\newblock On the determination of a function from spherical averages.
\newblock {\em SIAM J. Math. Anal.}, 19(1):214--232, 1988.

\bibitem{AER}
Yuri~A. Antipov, Ricardo Estrada, and Boris Rubin.
\newblock Method of analytic continuation for the inverse spherical mean
  transform in constant curvature spaces.
\newblock {\em J. Anal. Math.}, 118(2):623--656, 2012.

\bibitem{blais}
J.~A.~Rod Blais and Dean~A. Provins.
\newblock Spherical harmonic analysis and synthesis for global multiresolution
  applications.
\newblock {\em Journal of Geodesy}, 76(1):29--35, 2002.

\bibitem{book_L2norm}
William~L. Briggs, Van~Emden Henson, and Steve~F. McCormick.
\newblock {\em A multigrid tutorial}.
\newblock Society for Industrial and Applied Mathematics (SIAM), Philadelphia,
  PA, second edition, 2000.

\bibitem{Cheney-Borden-Book}
Margaret Cheney and Brett Borden.
\newblock {\em Fundamentals of radar imaging}, volume~79 of {\em CBMS-NSF
  Regional Conference Series in Applied Mathematics}.
\newblock Society for Industrial and Applied Mathematics (SIAM), Philadelphia,
  PA, 2009.

\bibitem{deHoop}
Maarten~V. de~Hoop.
\newblock Microlocal analysis of seismic inverse scattering.
\newblock In {\em Inside out: inverse problems and applications}, volume~47 of
  {\em Math. Sci. Res. Inst. Publ.}, pages 219--296. Cambridge Univ. Press,
  Cambridge, 2003.

\bibitem{FHR-Spherical-Even}
David Finch, Markus Haltmeier, and Rakesh.
\newblock Inversion of spherical means and the wave equation in even
  dimensions.
\newblock {\em SIAM J. Appl. Math.}, 68(2):392--412, 2007.

\bibitem{FPR-Spherical-Odd}
David Finch, Sarah~K. Patch, and Rakesh.
\newblock Determining a function from its mean values over a family of spheres.
\newblock {\em SIAM J. Math. Anal.}, 35(5):1213--1240 (electronic), 2004.

\bibitem{Finch-Rakesh}
David Finch and Rakesh.
\newblock The spherical mean value operator with centers on a sphere.
\newblock {\em Inverse Problems}, 23(6):S37--S49, 2007.

\bibitem{GGG-Book}
Israel~M. Gelfand, Simon~G. Gindikin, and Mark~I. Graev.
\newblock {\em Selected topics in integral geometry}, volume 220 of {\em
  Translations of Mathematical Monographs}.
\newblock American Mathematical Society, Providence, RI, 2003.
\newblock Translated from the 2000 Russian original by A. Shtern.

\bibitem{Golub-Kahan}
Gene Golub and William Kahan.
\newblock Calculating the singular values and pseudo-inverse of a matrix.
\newblock {\em J. Soc. Indust. Appl. Math. Ser. B Numer. Anal.}, 2:205--224,
  1965.

\bibitem{Haltmeier}
Markus Haltmeier.
\newblock Universal inversion formulas for recovering a function from spherical
  means.
\newblock {\em SIAM J. Math. Anal.}, 46(1):214--232, 2014.

\bibitem{Hansen-TSVD}
Per~Christian Hansen.
\newblock The truncated {SVD} as a method for regularization.
\newblock {\em BIT}, 27(4):534--553, 1987.

\bibitem{HKN}
Yulia Hristova, Peter Kuchment, and Linh Nguyen.
\newblock Reconstruction and time reversal in thermoacoustic tomography in
  acoustically homogeneous and inhomogeneous media.
\newblock {\em Inverse Problems}, 24(5):055006, 25, 2008.

\bibitem{John-Book}
Fritz John.
\newblock {\em Plane waves and spherical means applied to partial differential
  equations}.
\newblock Dover Publications, Inc., Mineola, NY, 2004.
\newblock Reprint of the 1955 original.

\bibitem{Kalf_Paper}
Hubert Kalf.
\newblock On the expansion of a function in terms of spherical harmonics in
  arbitrary dimensions.
\newblock {\em Bull. Belg. Math. Soc.}, 2, 1995.

\bibitem{Kuchment-Kunyansky_2008}
Peter Kuchment and Leonid Kunyansky.
\newblock Mathematics of thermoacoustic tomography.
\newblock {\em European J. Appl. Math.}, 19(2):191--224, 2008.

\bibitem{Kunyansky-Spherical-1}
Leonid~A. Kunyansky.
\newblock Explicit inversion formulae for the spherical mean {R}adon transform.
\newblock {\em Inverse Problems}, 23(1):373--383, 2007.

\bibitem{Kunyansky-Spherical-2}
Leonid~A. Kunyansky.
\newblock A series solution and a fast algorithm for the inversion of the
  spherical mean {R}adon transform.
\newblock {\em Inverse Problems}, 23(6):S11--S20, 2007.

\bibitem{Lin-Pinkus}
Vladimir~Ya. Lin and Allan Pinkus.
\newblock Fundamentality of ridge functions.
\newblock {\em J. Approx. Theory}, 75(3):295--311, 1993.

\bibitem{Linz-Book}
Peter Linz.
\newblock {\em Analytical and numerical methods for {V}olterra equations},
  volume~7 of {\em SIAM Studies in Applied Mathematics}.
\newblock Society for Industrial and Applied Mathematics (SIAM), Philadelphia,
  PA, 1985.

\bibitem{Louis-Quinto}
Alfred~K. Louis and Eric~Todd Quinto.
\newblock Local tomographic methods in sonar.
\newblock In {\em Surveys on solution methods for inverse problems}, pages
  147--154. Springer, Vienna, 2000.

\bibitem{Mensah-Franceschini}
Serge {Mensah} and \'Emilie {Franceschini}.
\newblock {Near-field ultrasound tomography}.
\newblock {\em { J. Acoust. Soc. Am}}, 121(3-4):1423--1433, 2007.

\bibitem{Nguyen}
Linh~V. Nguyen.
\newblock A family of inversion formulas in thermoacoustic tomography.
\newblock {\em Inverse Probl. Imaging}, 3(4):649--675, 2009.

\bibitem{Norton-Circular}
Stephen~J. Norton.
\newblock Reconstruction of a two-dimensional reflecting medium over a circular
  domain: exact solution.
\newblock {\em J. Acoust. Soc. Amer.}, 67(4):1266--1273, 1980.

\bibitem{Norton-Linzer}
Stephen~J. Norton and Melvin Linzer.
\newblock Reconstructing spatially incoherent random sources in the nearfield:
  exact inversion formulas for circular and spherical arrays.
\newblock {\em J. Acoust. Soc. Amer.}, 76(6):1731--1736, 1984.

\bibitem{PlatoArticle}
Robert Plato.
\newblock The regularizing properties of the composite trapezoidal method for
  weakly singular {V}olterra integral equations of the first kind.
\newblock {\em Adv. Comput. Math.}, 36(2):331--351, 2012.

\bibitem{Polyanin}
Andrei~D. Polyanin and Alexander~V. Manzhirov.
\newblock {\em Handbook of integral equations}.
\newblock Chapman \& Hall/CRC, Boca Raton, FL, second edition, 2008.

\bibitem{Quinto-Spherical-2}
Eric~Todd Quinto.
\newblock Support theorems for the spherical {R}adon transform on manifolds.
\newblock {\em Int. Math. Res. Not.}, pages Art. ID 67205, 17, 2006.

\bibitem{RKCV}
Souvik Roy, Venkateswaran~P. Krishnan, Praveen Chandrashekar, and
  A.~S.~Vasudeva Murthy.
\newblock An efficient numerical algorithm for the inversion of an integral
  transform arising in ultrasound imaging.
\newblock {\em J. Math. Imaging Vision}, 53(1):78--91, 2015.

\bibitem{Rubin}
Boris Rubin.
\newblock Inversion formulae for the spherical mean in odd dimensions and the
  {E}uler-{P}oisson-{D}arboux equation.
\newblock {\em Inverse Problems}, 24(2):025021, 10, 2008.

\bibitem{Salman}
Yehonatan Salman.
\newblock An inversion formula for the spherical mean transform with data on an
  ellipsoid in two and three dimensions.
\newblock {\em J. Math. Anal. Appl.}, 420(1):612--620, 2014.

\bibitem{Stefanov-Uhlmann-TAT}
Plamen Stefanov and Gunther Uhlmann.
\newblock Thermoacoustic tomography with variable sound speed.
\newblock {\em Inverse Problems}, 25(7):075011, 16, 2009.

\bibitem{Stefanov-Uhlmann-brain-imaging}
Plamen Stefanov and Gunther Uhlmann.
\newblock Thermoacoustic tomography arising in brain imaging.
\newblock {\em Inverse Problems}, 27(4):045004, 26, 2011.

\bibitem{Tricomi}
Francesco~G. Tricomi.
\newblock {\em Integral equations}.
\newblock Dover Publications, Inc., New York, 1985.
\newblock Reprint of the 1957 original.

\bibitem{Volterra-book}
Vito Volterra.
\newblock {\em Theory of functionals and of integral and integro-differential
  equations}.
\newblock With a preface by G. C. Evans, a biography of Vito Volterra and a
  bibliography of his published works by E. Whittaker. Dover Publications,
  Inc., New York, 1959.

\bibitem{Weiss_Product_Integration_Paper}
Richard Weiss.
\newblock Product integration for the generalized {A}bel equation.
\newblock {\em Math. Comp.}, 26:177--190, 1972.

\bibitem{Xu-Wang}
Minghua Xu and Lihong~V. Wang.
\newblock Time-domain reconstruction for thermoacoustic tomography in a
  spherical geometry.
\newblock {\em IEEE Transactions on Medical Imaging}, 21(7):814--822, July
  2002.

\end{thebibliography}

\end{document}